\documentclass[12pt]{amsart}
\usepackage{amsfonts,amsthm,amsmath,amscd,amssymb}
\usepackage[normalem]{ulem}
\newcommand{\stkout}[1]{\ifmmode\text{\sout{\ensuremath{#1}}}\else\sout{#1}\fi}

\providecommand{\xto}[1]{\xrightarrow{#1}}

\usepackage[T1]{fontenc}
\usepackage[mathscr]{eucal}
\usepackage{indentfirst}
\usepackage{graphicx}
\usepackage{graphics}
\usepackage[dvipsnames]{xcolor}
\usepackage{pict2e}
\usepackage{epic}
\usepackage{url}

\usepackage{breakurl}
\usepackage{hyperref}
\numberwithin{equation}{section}
\usepackage[margin=2.9cm]{geometry}
\usepackage{epstopdf} 
\usepackage{parskip}
\usepackage{array}
\usepackage{booktabs}
\usepackage{enumerate}
\usepackage{adjustbox}
\usepackage{diagbox}
\usepackage{tikz}
\usepackage{tikz-cd}
\usetikzlibrary{positioning,shadows,backgrounds,shapes.geometric}
\usepackage{mdframed}
\usepackage{float}

\theoremstyle{definition}
\newtheorem{thm}{Theorem}[section]

\newtheorem{cor}[thm]{Corollary}
\newtheorem{prop}[thm]{Proposition}

 \theoremstyle{definition}
\newtheorem{defn}[thm]{Definition}

\newtheorem{definition}[thm]{Definition}
\newtheorem{remark}[thm]{Remark}
\newtheorem{lemma}[thm]{Lemma}

\newtheorem{example}[thm]{Example}

\newtheorem{proposition}[thm]{Proposition}
\newtheorem{corollary}[thm]{Corollary}

\newtheorem{theorem}[thm]{Theorem}

\newcommand{\thistheoremname}{}
\newtheorem{genericthm}[thm]{\thistheoremname}

\newcommand{\Hom}{{\rm{Hom}}}

\newcommand{\floor}[1]{\lfloor{#1}\rfloor}

\newcommand{\id}{\operatorname{id}}

\newcommand{\Char}{\operatorname{char}}

\newcommand{\Spec}{\operatorname{Spec}}

\newcommand{\GW}{\operatorname{GW}}
\newcommand{\Gr}{\operatorname{Gr}}

\newcommand{\ind}{\operatorname{ind}}
\newcommand{\sgn}{\operatorname{sgn}}

\newcommand{\Jac}{\operatorname{Jac}}

\providecommand{\Bez}{\text{B\'{e}z}}
\providecommand{\KO}{\text{KO}}
\newcommand{\fg}{\text{f.g.}}
\newcommand{\HH}{\mathbb{H}}
\newcommand{\Mod}{\ \mathrm{mod}\ }

\newcommand{\pcoor}[1]{%
  \begingroup\lccode`~=`: \lowercase{\endgroup
  \edef~}{\mathbin{\mathchar\the\mathcode`:}\nobreak}%
  [
  \begingroup
  \mathcode`:=\string"8000
  #1%
  \endgroup 
  ]
}

\newcommand{\til}[1]{\widetilde{#1}}
\newcommand{\gw}[1]{\left\langle#1\right\rangle}

\newcommand{\sheafHom}{\mathcal{H}\hspace{-0.1em}\mathit{om}}

\providecommand{\A}{\mathbb{A}}
\providecommand{\F}{\mathbb{F}}
\renewcommand{\P}{\mathbb{P}}

\usepackage[style=alphabetic,maxbibnames=9]{biblatex}
\bibliography{sources}
\DeclareFieldFormat{postnote}{#1}
\DeclareFieldFormat{multipostnote}{#1}
\providecommand{\Alg}{\texttt{Alg}}

\begin{document}
\title{B\'ezoutians and the $\mathbb{A}^1$-Degree}

\author[Brazelton]{Thomas Brazelton}
\author[McKean]{Stephen McKean}
\author[Pauli]{Sabrina Pauli}

\begin{abstract}
    We prove that both the local and global $\mathbb{A}^1$-degree of an endomorphism of affine space can be computed in terms of the multivariate B\'ezoutian. In particular, we show that the B\'ezoutian bilinear form, the Scheja--Storch form, and the $\mathbb{A}^1$-degree for complete intersections are isomorphic. Our global theorem generalizes Cazanave's theorem in the univariate case, and our local theorem generalizes Kass--Wickelgren's theorem on EKL forms and the local degree. This result provides an algebraic formula for local and global degrees in motivic homotopy theory.
\end{abstract}
\maketitle

\section{Introduction}
Morel's $\A^1$-Brouwer degree~\cite{Morel} assigns a bilinear form-valued invariant to a given endomorphism of affine space. However, Morel's construction is not explicit. In order to make computations and applications, we would like algebraic formulas for the $\A^1$-degree. Such formulas were constructed by Cazanave for the global $\A^1$-degree in dimension 1~\cite{Cazanave}, Kass--Wickelgren for the local $\A^1$-degree at rational points and \'etale points~\cite{KW-EKL}, and Brazelton--Burklund--McKean--Montoro--Opie for the local $\A^1$-degree at separable points~\cite{trace-paper}. In this paper, we give a general algebraic formula for the $\A^1$-degree in both the global and local cases. In the global case, we remove Cazanave's dimension restriction, while in the local case, we remove previous restrictions on the residue field of the point at which the local $\A^1$-degree is taken.

Let $k$ be a field, and let $f = (f_1,\ldots,f_n) : \A^n_k \to \A^n_k$ be an endomorphism of affine space with isolated zeros, so that $Q := k[x_1,\ldots,x_n]/(f_1,\ldots,f_n)$ is a complete intersection. We now recall the definition of the B\'ezoutian of $f$, as well as a special bilinear form determined by the B\'ezoutian. Introduce new variables $X:=(X_1,\ldots,X_n)$ and $Y:=(Y_1,\ldots,Y_n)$. For each $1\le i,j\le n$, define the quantity
\begin{align*}
    \Delta_{ij} := \frac{f_i(Y_1, \ldots, Y_{j-1},X_j, \ldots, X_n) - f_i \left( Y_1, \ldots, Y_j, X_{j+1}, \ldots, X_n \right)}{X_j - Y_j}.
\end{align*}

\begin{defn}\label{def:Bezoutian-intro-def}
The \textit{B\'ezoutian} of $f$ is the image $\Bez(f_1, \ldots, f_n)$ of the determinant $\det \left( \Delta_{ij} \right)$ in $k[X,Y]/(f(X),f(Y))$. Given a basis $\{a_1,\ldots,a_m\}$ of $Q$ as a $k$-vector space, there exist scalars $B_{i,j}$ for which
\[\Bez(f_1,\ldots,f_n)=\sum_{i,j=1}^m B_{i,j}a_i(X)a_j(Y).\]
We define the \textit{B\'ezoutian form} of $f$ to be the class $\beta_f$ in the Grothendieck--Witt ring $\GW(k)$ determined by the bilinear form $Q\times Q\to k$ with Gram matrix $(B_{i,j})$.
\end{defn}

For any isolated zero of $f$ corresponding to a maximal ideal $\mathfrak{m}$, there is an analogous bilinear form $\beta_{f,\mathfrak{m}}$ on the local algebra $Q_\mathfrak{m}$. We refer to $\beta_{f,\mathfrak{m}}$ as the \textit{local B\'ezoutian form} of $f$ at $\mathfrak{m}$. We will demonstrate that both $\beta_f$ and $\beta_{f,\mathfrak{m}}$ yield well-defined classes in $\GW(k)$. Our main theorem is that the B\'ezoutian form of $f$ agrees with the $\A^1$-degree in both the local and global contexts.

\begin{theorem}\label{thm:main-thm} 
Let $\Char{k}\neq 2$. Let $f: \A^n_k \to \A^n_k$ have an isolated zero at a closed point $\mathfrak{m}$. Then $\beta_{f,\mathfrak{m}}$ is isomorphic to the local $\A^1$-degree of $f$ at $\mathfrak{m}$. If we further assume that all the zeros of $f$ are isolated, then $\beta_f$ is isomorphic to the global $\A^1$-degree of $f$.
\end{theorem}

Because the B\'ezoutian form can be explicitly computed using commutative algebraic tools, Theorem~\ref{thm:main-thm} provides a tractable formula for $\A^1$-degrees and Euler classes in motivic homotopy theory. Using the B\'ezoutian formula for the $\A^1$-degree, we are able to deduce several computational rules for the degree. We also provide a Sage implementation for calculating local and global $\A^1$-degrees via the B\'ezoutian at~\cite{code}.

\begin{remark}
The key contribution of this article is computability. Building on the work of Kass--Wickelgren~\cite{KW-EKL}, Bachmann--Wickelgren~\cite{BW} show that the $\A^1$-degree agrees with the Scheja--Storch form as elements of $\KO^0(k)$. In Theorem~\ref{thm:deg-is-SS}, we show how this immediately implies that the $\A^1$-degree and Scheja--Storch form determine the same element of $\GW(k)$. Scheja--Storch~\cite{SchejaStorch} showed that their form is a B\'ezoutian bilinear form (in the sense of Definition~\ref{def:Bezoutian-bilinear-form}; see also Lemma~\ref{lem:Bezoutian-equals-SS} and Remark~\ref{rem:local-version}), which was further explored by Becker--Cardinal--Roy--Szafraniec~\cite{BCRS}. Putting these results together shows that the isomorphism class of the B\'ezoutian bilinear form is the $\A^1$-degree.

In dimension 1, Cazanave~\cite{Cazanave} gives a simple formula for computing the $\A^1$-degree as a B\'ezoutian bilinear form in the global setting. However, it is not immediately clear how to adapt this to higher dimensions or the local setting. Becker--Cardinal--Roy--Szafraniec show how to compute B\'ezoutian bilinear forms in terms of ``dualizing forms,'' but this method is computationally analogous to using the Eisenbud--Khimshiashvili--Levine form to compute the $\A^1$-degree~\cite{KW-EKL}. In the proof of Theorem~\ref{thm:main-thm} (found in Section~\ref{sec:proof-of-main-theorem}), we show that our two notions of B\'ezoutian bilinear forms (Definitions~\ref{def:Bezoutian-intro-def} and~\ref{def:Bezoutian-bilinear-form}) agree up to isomorphism. Since Definition~\ref{def:Bezoutian-intro-def} is the desired generalization of Cazanave's formula, this enables us to calculate $\A^1$-degrees in full generality.
\end{remark}

\subsection{Outline}
Before proving Theorem~\ref{thm:main-thm}, we recall some classical results on B\'ezoutians (following~\cite{BCRS}) in Section~\ref{sec:Bezoutians}, as well as the work of Scheja--Storch on residue pairings~\cite{SchejaStorch} in Section~\ref{sec:Scheja-Storch}. We then discuss a local decomposition procedure for the Scheja--Storch form and show that the global Scheja--Storch form is isomorphic to the B\'ezoutian form in Section~\ref{sec:local-to-global}. In Section~\ref{sec:proof-of-main-theorem}, we complete the proof of Theorem~\ref{thm:main-thm} by applying the work of Kass--Wickelgren~\cite{KW-EKL} and Bachmann--Wickelgren~\cite{BW} on the local $\A^1$-degree and the Scheja--Storch form. Using Theorem~\ref{thm:main-thm}, we give an algorithm for computing the local and global $\A^1$-degree at the end of Section~\ref{sec:computing algorithm}, available at~\cite{code}. In Section~\ref{sec:calculation-rules}, we establish some basic properties for computing degrees. In Section~\ref{sec:examples}, we provide a step-by-step illustration of our ideas by working through some explicit examples. Finally, we implement our code to compute some examples of $\A^1$-Euler characteristics of Grassmannians in Section~\ref{sec:Gr}. We check our computations by proving a general formula for the $\A^1$-Euler characteristic of a Grassmannian in Theorem~\ref{thm:grassmannian}. The $\A^1$-Euler characteristic of Grassmannians is essentially a folklore result that follows from the work of Hoyois, Levine, and Bachmann--Wickelgren.

\subsection{Background}
Let $\GW(k)$ denote the Grothendieck--Witt group of isomorphism classes of symmetric, non-degenerate bilinear forms over a field $k$. Morel's $\A^1$-Brouwer degree~\cite[Corollary 1.24]{Morel}
\[\deg:[\P^n_k/\P^{n-1}_k,\P^n_k/\P^{n-1}_k]_{\A^1}\to\GW(k),\]
which is a group isomorphism (in fact, a ring isomorphism~\cite[Lemma 6.3.8]{Morel-Trieste}) for $n\geq 2$, demonstrates that bilinear forms play a critical role in motivic homotopy theory. However, Morel's $\A^1$-degree is non-constructive. Kass and Wickelgren addressed this problem by expressing the $\A^1$-degree as a sum of local degrees~\cite[Lemma~19]{KW-cubic} and providing an explicit formula (building on the work of Eisenbud--Levine~\cite{EisenbudLevine} and Khimshiashvili~\cite{Khim}) for the local $\A^1$-degree~\cite{KW-EKL} at rational points and \'etale points. This explicit formula can also be used to compute the local $\A^1$-degree at points with separable residue field by~\cite{trace-paper}. Together, these results allow one to compute the global $\A^1$-degree of a morphism $f:\A^n_k\to\A^n_k$ with only isolated zeros by computing the local $\A^1$-degrees of $f$ over its zero locus, so long as the residue field of each point in the zero locus is separable over the base field. In the local case, Theorem~\ref{thm:main-thm} gives a commutative algebraic formula for the local $\A^1$-degree at any closed point.

Cazanave showed that the B\'ezoutian gives a formula for the global $\A^1$-degree of any endomorphism of $\P^1_k$~\cite{Cazanave}. An advantage to Cazanave's formula is that one does not need to determine the zero locus or other local information about $f$. We extend Cazanave's formula for morphisms $f:\A^n_k\to\A^n_k$ with isolated zeros. The work of Scheja--Storch on global complete intersections~\cite{SchejaStorch} is central to both~\cite{KW-EKL} and our result. We also rely on the work of Becker--Cardinal--Roy--Szafraniec~\cite{BCRS}, who describe a procedure for recovering the global version of the Scheja--Storch form.

Theorem~\ref{thm:main-thm} has applications wherever Morel's $\A^1$-degree is used. One particularly successful application of the $\A^1$-degree has been the $\A^1$-enumerative geometry program. The goal of this program is to enrich enumerative problems over arbitrary fields by producing $\GW(k)$-valued enumerative equations and interpreting them geometrically over various fields. Notable results in this direction include Srinivasan and Wickelgren's count of lines meeting four lines in three-space \cite{SW}, Larson and Vogt's count of bitangents to a smooth plane quartic \cite{LV}, and Bethea, Kass, and Wickelgren's enriched Riemann--Hurwitz formula \cite{BKW}. See~\cite{McKean,Pauli} for other related works. For a more detailed account of recent developments in $\mathbb{A}^1$-enumerative geometry, see \cite{Brazelton-expos,Pauli-expos}.

\subsection*{Acknowledgements}
We thank Tom Bachmann, Gard Helle, Kyle Ormsby, Paul Arne \O stv\ae r, and Kirsten Wickelgren for helpful comments. We thank Kirsten Wickelgren for suggesting that we make available a code implementation for computing local and global $\A^1$-degrees. Thank you to Herman Rohrbach for pointing out a typo. The first named author is supported by an NSF Graduate Research Fellowship (DGE-1845298). The second named author received support from Kirsten Wickelgren's NSF CAREER grant (DMS-1552730). The last named author acknowledges support from a project that has received funding from the European Research Council (ERC) under the European Union's Horizon 2020 research and innovation program (grant agreement No 832833).

\section{Notation and conventions}

In this section, we fix some standard terminology and notation. Let $k$ denote an arbitrary field. We will always use $f = (f_1, \ldots, f_n) : \A^n_k \to \A^n_k$ to denote an endomorphism of affine space, assumed to have isolated zeros when we work with it in the global context. We denote by $Q$ the global algebra associated to this endomorphism
\begin{align*}
    Q := \frac{k[x_1, \ldots, x_n]}{(f_1, \ldots, f_n)}.
\end{align*}
For any maximal ideal $\mathfrak{m}$ of $Q$ on which $f$ vanishes, we denote by $Q_\mathfrak{m}$ the local algebra
\begin{align*}
    Q_\mathfrak{m} := \frac{k[x_1, \ldots, x_n]_{\mathfrak{m}}}{(f_1, \ldots, f_n)}.
\end{align*}
If $\lambda : V \to k$ is a $k$-linear form on any $k$-algebra, we will denote by $\Phi_\lambda$ the associated bilinear form given by
\begin{align*}
    \Phi_\lambda : V \times V &\to k \\
    (a,b) &\mapsto \lambda(ab).
\end{align*}

\begin{defn}
We say that $\lambda$ is a \textit{dualizing linear form} if $\Phi_\lambda$ is non-degenerate as a symmetric bilinear form \cite[2.1]{BCRS}. If $\lambda$ is dualizing, then we say that two vector space bases $\{a_i\}$ and $\{b_i\}$ of $V$ are \textit{dual with respect to} $\lambda$ if
\begin{align*}
    \lambda(a_i b_j) = \delta_{ij},
\end{align*}
where $\delta_{ij}=1$ for $i=j$ and $\delta_{ij}=0$ for $i\neq j$. We show in Remark~\ref{rem:dual vs dualizing} that if $\{a_i\}$ and $\{b_i\}$ are dual with respect to $\lambda$, then $\lambda$ is a dualizing linear form.
\end{defn}

More notation will be introduced as we provide an overview of B\'{e}zoutians and the Scheja--Storch bilinear form. We will borrow and clarify notation from both \cite{SchejaStorch} and \cite{BCRS}.

\section{B\'ezoutians}\label{sec:Bezoutians}

We first provide an overview of the construction of the \textit{B\'ezoutian}, following \cite{BCRS}. Given one of our $n$ polynomials $f_i$, we introduce two sets of auxiliary indeterminants and study how $f_i$ changes when we incrementally exchange one set of indeterminants for the other. Explicitly, consider variables $X:=(X_1,\ldots,X_n)$ and $Y:=(Y_1,\ldots,Y_n)$. For any $1\leq i,j\leq n$, we denote by $\Delta_{ij}$ the quantity
\begin{align*}
    \Delta_{ij} := \frac{f_i(Y_1, \ldots, Y_{j-1},X_j, \ldots, X_n) - f_i \left( Y_1, \ldots, Y_j, X_{j+1}, \ldots, X_n \right)}{X_j - Y_j}.
\end{align*}
We view this as living in the tensor product ring $Q \otimes_k Q$, under the isomorphism
\begin{align*}
    \varepsilon:\frac{k[X,Y]}{(f(X),f(Y))} \xto{\cong} Q \otimes_k Q,
\end{align*}
given by sending $X_i$ to $x_i \otimes 1$, and $Y_i$ to $1\otimes x_i$.
\begin{definition}\label{def:Bezoutian} We define the \textit{B\'ezoutian} of the polynomials $f_1,\ldots,f_n$ to be the image $\Bez(f_1,\ldots,f_n)$ of the determinant $\det\left( \Delta_{ij} \right)$ in $Q \otimes_k Q$.
\end{definition}

\begin{example}\label{ex:a3-squared-map} Let $(f_1,f_2,f_3) = (x_1^2, x_2^2, x_3^2)$. Then we have that
\begin{align*}
    \Bez(f_1,f_2,f_3) &= \varepsilon\left(\det \begin{pmatrix} X_1 + Y_1 & 0 & 0 \\ 0 & X_2 + Y_2 & 0 \\ 0 & 0 & X_3 + Y_3 \end{pmatrix}\right) \\
    &= \varepsilon\left((X_1 + Y_1)(X_2 + Y_2)(X_3 + Y_3)\right)\\
    &= x_1x_2x_3\otimes 1+x_1x_2\otimes x_3+x_1x_3\otimes x_2+x_2x_3\otimes x_1\\
    &+x_1\otimes x_2x_3+x_2\otimes x_1x_3+ x_3\otimes x_1x_2+1\otimes x_1x_2x_3.
\end{align*}
\end{example}

There is a natural multiplication map $\delta : Q \otimes_k Q \to Q$, defined by $\delta(a\otimes b)=ab$, that sends the B\'ezoutian of $f$ to the image of the Jacobian of $f$ in $Q$.

\begin{proposition} 
Let $\Jac(f_1,\ldots,f_n)$ be the image of the Jacobian determinant $\det(\frac{\partial f_i}{\partial x_j})$ in $Q$. Then
\begin{align*}
    \delta \left( \Bez(f_1, \ldots, f_n) \right)&= \Jac(f_1, \ldots, f_n) \in Q.
\end{align*}
\end{proposition}
\begin{proof} 
Note that $(\delta\circ\varepsilon)(a(X,Y))=a(x,x)$ and $\delta\circ\varepsilon$ is an algebra homomorphism. In particular, $\delta\circ\varepsilon$ preserves the multiplication and addition occurring in the determinant which defines $\Bez(f_1, \ldots, f_n)$. Therefore it suffices for us to verify that
\begin{align*}
    (\delta\circ\varepsilon) \left( \Delta_{ij} \right) = \frac{\partial f_i}{\partial x_j}.
\end{align*}
Recall that
\begin{align*}
    \Delta_{ij} &= \frac{f_i(Y_1, \ldots, Y_{j-1}, X_j, \ldots, X_n) - f_i (Y_1, \ldots, Y_j, X_{j+1}, \ldots, X_n)}{X_j - Y_j}.
\end{align*}
To understand $(\delta\circ\varepsilon)(\Delta_{ij})$, we first evaluate $X_\ell\mapsto x_\ell$ and $Y_\ell\mapsto x_\ell$ for $\ell\neq j$. Since $\Delta_{ij}$ is a polynomial, it follows that we may write
\begin{align*}
    f_i(x_1, \ldots, X_j, \ldots, x_n) - f_i(x_1, \ldots, Y_j, \ldots, x_n) &= g_i(x_1,\ldots, X_j,\ldots,x_n,Y_j) \cdot (X_j - Y_j)
\end{align*}
for some polynomial $g_i$. Derive both sides of the equality above with respect to $X_j$. Then we have that
\begin{align*}
    \frac{\partial f_i}{\partial X_j} &= \frac{\partial g_i}{\partial X_j}(X_j - Y_j) + g_i(x_1,\ldots, X_j,\ldots,x_n,Y_j).
\end{align*}
We may now apply $\delta\circ\varepsilon$. The left hand side remains unchanged, while the first term on the right hand side vanishes, leaving us with
\begin{align*}
    \frac{\partial f_i}{\partial x_j} &= g_i(x_1,\ldots,x_j,\ldots,x_n,x_j) = (\delta\circ\varepsilon) \left( \Delta_{ij} \right).\qedhere
\end{align*}
\end{proof}

\begin{lemma}
\label{lemma:bezoutian-2-bases}
Let $a_1, \ldots, a_m$ be any vector space basis for $Q$, and write the B\'{e}zoutian as
\begin{align*}
    \Bez(f_1, \ldots, f_n) = \sum_{i=1}^m a_i \otimes b_i
\end{align*}
for some $b_1,\ldots,b_n\in Q$. 
Then $\left\{ b_i \right\}_{i=1}^m$ is a basis for $Q$ . 
\end{lemma}
\begin{proof}
This is \cite[2.10(iii)]{BCRS}.
\end{proof}
This allows us to associate to the B\'ezoutian a pair of vector space bases for $Q$. Given any such pair of bases, we will construct a unique linear form for which the bases are dual. Before doing so, we establish some equivalent conditions for the duality of a linear form given a pair of bases.

\begin{proposition}\label{prop:equivalent-conditions-bases-dual}
Let $\left\{ a_i \right\}$ and $\left\{ b_i \right\}$ be a pair of bases for $B$. Consider the induced $k$-linear isomorphism
\begin{align*}
    \Theta: \Hom_k(Q,k) &\to Q \\
    \varphi &\mapsto \sum_i \varphi(a_i)b_i.
\end{align*}
Given a linear form $\lambda: Q \to k$, the following are equivalent:
\begin{enumerate}
    \item We have that $\Theta(\lambda) = \sum_i \lambda(a_i)b_i = 1$.
    \item For any $a\in Q$, we have $a = \sum_i \lambda(aa_i)b_i$.
    \item We have that $\left\{ a_i \right\}$ and $\left\{ b_i \right\}$ are dual with respect to $\lambda$.
\end{enumerate}

\end{proposition}
\begin{proof} 
Note that (2) implies (1) by setting $a=1$. Next, we remark that $\Theta$ is a $Q$-module isomorphism by \cite[3.3 Satz]{SchejaStorch}, where the $Q$-module structure on $\Hom_k(Q,k)$ is given by $a\cdot \varphi = \varphi(a\cdot -)$. This allows us to conclude that $a\cdot \Theta(\lambda) = \Theta(a\cdot \lambda)$ for any linear form $\lambda$. In particular, we have
\begin{align*}
    a \sum_i \lambda(a_i)b_i = \sum_i \lambda(aa_i)b_i.
\end{align*}
It follows from this identity that (1) implies (2). Now suppose that (2) holds. By setting $a=b_j$ for some $j$, we have
\begin{align*}
    \sum_i\lambda(a_ib_j)b_i=b_j.
\end{align*}
Since $\{b_i\}$ is a basis, it follows that $\lambda(a_ib_j) = \delta_{ij}$. Thus the bases $\{a_i\}$ and $\{b_i\}$ are dual with respect to $\lambda$. Finally, suppose that (3) holds, so that $\lambda(a_ib_j)=\delta_{ij}$. For any $a\in Q$, write $a$ as $a := \sum_j c_j b_j$ for some scalars $c_j$. Then
\begin{align*}
    \sum_i \lambda (aa_i)b_i &= \sum_i \lambda \left( a_i\sum_j c_j b_j \right)b_i = \sum_i \left( \sum_j c_j \lambda(a_ib_j) \right)b_i \\
    &= \sum_i c_i b_i = a.
\end{align*}
Thus (3) implies (2).
\end{proof}

\begin{remark}\label{rem:dual vs dualizing}
If $\{a_i\}$ and $\{b_i\}$ are dual with respect to $\lambda$, then $\lambda$ is a dualizing form. Indeed, suppose there exists $x\in Q$ such that $\Phi_\lambda(x,y)=0$ for all $y\in Q$. Write $x=\sum_i x_ia_i$ with $x_i\in k$. Then 
\begin{align*}
    0&=\lambda(xb_j)=\lambda\left(\sum_ix_ia_ib_j\right)\\
    &=\sum_ix_i\lambda(a_ib_j)=x_j
\end{align*}
for all $j$, so $x=0$.
\end{remark}

As $\Theta$ is a $k$-algebra isomorphism, it admits a unique preimage of 1. Thus, given any pair of bases $\{a_i\}$ and $\{b_i\}$ of $Q$, there is a unique dualizing linear form with respect to which $\{a_i\}$ and $\{b_i\}$ are dual.

\begin{cor} \label{cor:existence-dualizing-linear-form}
Let $\{a_i\}$ and $\{b_i\}$ be two $k$-vector space bases for $Q$. Then there exists a unique dualizing linear form $\lambda:Q\to k$ such that $\{a_i\}$ and $\{b_i\}$ are dual with respect to $\lambda$.
\end{cor}

\begin{definition}\label{def:Bezoutian-bilinear-form} 
We call $\Phi_\lambda$ a \textit{B\'ezoutian bilinear form} if $\lambda: Q \to k$ is a dualizing linear form such that
\begin{align*}
    \Bez(f_1, \ldots, f_n) = \sum_{i=1}^m a_i \otimes b_i,
\end{align*}
where $\{a_i\}$ and $\{b_i\}$ are dual bases with respect to $\lambda$.
\end{definition}
A priori this is different than the B\'ezoutian form detailed in Definition~\ref{def:Bezoutian-intro-def}, although we will prove that they define the same class in $\GW(k)$ in Section~\ref{sec:computing algorithm}.

\begin{proposition} Given a function $f : \A^n_k \to \A^n_k$ with isolated zeros, its B\'{e}zoutian bilinear form is a well-defined class in $\GW(k)$.
\end{proposition}
\begin{proof}
Let $\Phi_\lambda$ be a B\'ezoutian bilinear form for $f$. Recall that $\Phi_\lambda:Q\times Q\to k$ is defined by $\Phi_\lambda(a,b)=\lambda(ab)$. Since $\lambda$ is a dualizing linear form, $\Phi_\lambda$ is non-degenerate and  as $Q$ is commutative, $\Phi_\lambda$ is symmetric. Lemma \ref{lemma:bezoutian-2-bases} implies that given a basis $a_1,\ldots, a_m$ for $Q$, we can write 
\[\Bez(f_1,\ldots,f_n)=\sum_{i=1}^m a_i\otimes b_i,\]
and obtain a second basis $b_1,\ldots, b_m$ for $Q$. By Corollary \ref{cor:existence-dualizing-linear-form}, there is a dualizing linear form for the two bases $\{a_i\}_{i=1}^m$ and $\{b_i\}_{i=1}^m$. It remains to show that if 
\[\Bez(f_1,\ldots,f_n)=\sum_{i=1}^m a_i\otimes b_i=\sum_{i=1}^m a_i'\otimes b_i',\]
for some bases $\{a_i\},\{b_i\}$ dual with respect to $\lambda$ and $\{a_i'\},\{b_i'\}$ dual with respect to $\lambda'$, then $\Phi_\lambda$ and $\Phi_{\lambda'}$ are isomorphic. We will in fact show that $\lambda=\lambda'$, so that $\Phi_\lambda=\Phi_{\lambda'}$. 
Write $a_i=\sum_s\alpha_{is}a_s'$ and $b_i=\sum_s\beta_{is}b_s'$. Then
\begin{align*}
    \sum_{i=1}^m a_i'\otimes b_i'&=\sum_{i=1}^m a_i\otimes b_i =\sum_i\left(\sum_s\alpha_{is}a_s'\right)\otimes\left(\sum_t\beta_{it}b_t'\right)\\
    &=\sum_{s,t}\left(\sum_i\alpha_{is}\beta_{it}\right)a_s'\otimes b_t'.
\end{align*}
Since $\{a_s'\otimes b_t'\}$ is a basis for $Q\otimes_k Q$, we conclude that $\sum_i\alpha_{is}\beta_{it}=\delta_{st}$. In particular, $(\alpha_{ij})^{-1}=(\beta_{ij})^T$, so $(\beta_{ij})(\alpha_{ij})^T$ is the identity matrix. Thus $\sum_j\alpha_{sj}\beta_{tj}=\delta_{st}$.

Now given $g=\sum_i c_ia_i=\sum_i c_i'a_i'\in Q$ and $1=\sum_i d_ib_i=\sum_i d_i'b_i'$, we have that
\begin{align*}
    \lambda(g)=\lambda\left(\sum_i \left(c_ia_i\right)\cdot \sum_j d_j b_j\right)=\sum_{i,j} c_i d_j\lambda\left(a_i b_j\right)=\sum_i c_i d_i.
\end{align*}
Similarly, we have $\lambda'(g)=\sum_i c_i'd_i'.$ By our change of bases, we have $c_j'=\sum_i c_j\alpha_{ij}$ and $d_j'=\sum_i d_i\beta_{ij}$. Thus
\begin{align*}
    \lambda'(g)&=\sum_j c_j'd_j' =\sum_j\left(\sum_s c_s\alpha_{sj}\right)\left(\sum_t d_t\beta_{tj}\right)\\
    &=\sum_{s,t}c_sd_t\left(\sum_j\alpha_{sj}\beta_{tj}\right)=\sum_s c_sd_s=\lambda(g).
\end{align*}
Therefore $\lambda=\lambda'$, as desired. 
\end{proof}

\begin{example} Continuing Example~\ref{ex:a3-squared-map}, let $f = (x_1^2,x_2^2,x_3^2)$, so that
\begin{align*}
    \varepsilon^{-1}(\Bez(f_1, f_2, f_3)) &= (X_1 + Y_1)(X_2 + Y_2)(X_3 + Y_3) \\
    &= X_1 X_2 X_3 + X_1 X_2 Y_3 + X_1 Y_2 X_3 + X_1 Y_2 Y_3 \\
    &\quad + Y_1 X_2 X_3 + Y_1 X_2 Y_3 + Y_1 Y_2 X_3 + Y_1 Y_2 Y_3.
\end{align*}
We give two bases for $k[Z_1, Z_2, Z_3]/(Z_1^2, Z_2^2, Z_3^2)$ in the following table, where we replace $Z$ by either $X$ or $Y$. We pair off these bases in a convenient way.
\begin{center}
    \begin{tabular}{l | l  l}
    $i$ & $a_i$ & $b_i$ \\
    \hline
    1 & 1 & $Y_1 Y_2 Y_3$ \\
    2 & $X_1$ & $Y_2 Y_3$ \\
    3 & $X_2$ & $Y_1 Y_3$ \\
    4 & $X_3$ & $Y_1 Y_2$ \\
    5 & $X_1 X_2$ & $Y_3$ \\
    6 & $X_1 X_3$ & $Y_2$ \\
    7 & $X_2 X_3$ & $Y_1$ \\
    8 & $X_1 X_2 X_3$ & 1
    \end{tabular}
\end{center}
The B\'ezoutian we computed is in the desired form $\sum_{i=1}^8 a_i \otimes b_i$, so we now need to compute the dualizing linear form $\lambda$ for $\{a_i\}$ and $\{b_i\}$. Since $1=1\cdot b_8+\sum_{i=1}^7 0\cdot b_i$, we define $\lambda$ by $\lambda(a_i)=0$ for $1\leq i\leq 7$ and $\lambda(a_8)=\lambda(X_1X_2X_3)=1$. Now let $g \in k[X_1, X_2, X_3]/(X_1^2, X_2^2, X_3^2)$ be arbitrary. We can write $g$ as
\begin{align*}
    g &= c_1 + c_2 X_1 + c_3 X_2 + c_4 X_3 + c_5 X_1 X_2 + c_6 X_1 X_3 + c_7 X_2 X_3  + c_8 X_1 X_2 X_3.
\end{align*}
Then $\lambda$ is the dualizing linear form sending
\begin{align*}
    \lambda : \frac{k[X_1,X_2,X_3]}{(X_1^2, X_2^2, X_3^2)} &\to k \\
    g &\mapsto c_8.
\end{align*}
Finally we can compute the Gram matrix of $\Phi_\lambda$ in the basis $\{a_i\}$. Note that $a_ia_j$ is a scalar multiple of $X_1X_2X_3$ if and only if $i+j-1=8$. Thus the Gram matrix is
\begin{align*}
    \Phi_\lambda = \begin{pmatrix}
0 & 0 & 0 & 0 & 0 & 0 & 0 & 1\\
0 & 0 & 0 & 0 & 0 & 0 & 1 & 0\\
0 & 0 & 0 & 0 & 0 & 1 & 0 & 0\\
0 & 0 & 0 & 0 & 1 & 0 & 0 & 0\\
0 & 0 & 0 & 1 & 0 & 0 & 0 & 0\\
0 & 0 & 1 & 0 & 0 & 0 & 0 & 0\\
0 & 1 & 0 & 0 & 0 & 0 & 0 & 0\\
1 & 0 & 0 & 0 & 0 & 0 & 0 & 0
\end{pmatrix} \cong \bigoplus_{i=1}^4\begin{pmatrix}
1 & 0 \\ 0 & -1 \end{pmatrix}.
\end{align*}
\end{example}

\section{The Scheja--Storch bilinear form}\label{sec:Scheja-Storch}

Associated to any polynomial with an isolated zero, Eisenbud and Levine~\cite{EisenbudLevine} and Khimshiashvili~\cite{Khim} used the Scheja--Storch construction \cite{SchejaStorch} to produce a bilinear form on the local algebra $Q_\mathfrak{m}$. Kass and Wickelgren proved that this Eisenbud--Khimshiashvili--Levine bilinear form computes the local $\A^1$-degree \cite{KW-EKL}. The machinery of Scheja and Storch works in great generality; in particular, one may produce a Scheja--Storch bilinear form on the global algebra $Q$ as well as the local algebras $Q_\mathfrak{m}$. We will provide a brief account of the Scheja--Storch construction before comparing it with the B\'ezoutian.

In \cite{SchejaStorch}, $k\langle X\rangle:=k \left\langle X_1, \ldots, X_n \right\rangle$ denotes either a polynomial ring $k \left[ X_1, \ldots, X_n \right]$ or a power series ring $k \left[ \left[ X_1, \ldots, X_n \right] \right]$. We will also use this notation, although we will focus on the situation where $k\langle X\rangle$ is a polynomial ring. Let $\rho: k \left\langle X \right\rangle \to Q$ denote the map obtained by quotienting out by the ideal $\left( f_1, \ldots, f_n \right)$, let $\mu_1 : k \left\langle X \right\rangle \otimes_k k\left\langle X \right\rangle \to k \left\langle X \right\rangle$ denote the multiplication map, and let $\mu : Q \otimes_k Q \to Q$ denote the multiplication map on the global algebra, fitting into a commutative diagram
\[ \begin{tikzcd}
    k \left\langle X \right\rangle\otimes_k k \left\langle X \right\rangle\rar["\mu_1"]\dar["\rho \otimes \rho" left] & k \left\langle X \right\rangle\dar["\rho" right]\\
    Q \otimes_k Q\rar["\mu" below] & Q.
\end{tikzcd} \]
We remark that $f_j \otimes 1 - 1 \otimes f_j$ lies in $\ker(\mu_1)$, and that $\ker(\mu_1)$ is generated by elements of the form $X_i \otimes 1 - 1 \otimes X_i$. Thus for any $j$, there are elements $a_{ij}\in k \left\langle X \right\rangle \otimes_k k \left\langle X \right\rangle$ such that
\begin{equation}\label{eqn:aij}
\begin{aligned}
    f_j \otimes 1 - 1 \otimes f_j &= \sum_{i=1}^n a_{ij} \left( X_i \otimes 1 - 1 \otimes X_i \right).
\end{aligned}
\end{equation}
We denote by $\Delta$ the following distinguished element in the tensor algebra $Q \otimes_k Q$
\[\Delta := \left( \rho \otimes \rho \right) \left( \det (a_{ij}) \right),\]
which corresponds to the B\'ezoutian which we will later demonstrate. It is true that $\Delta$ is independent of the choice of $a_{ij}$, as shown by Scheja and Storch \cite[3.1 Satz]{SchejaStorch}. We now define an important isomorphism $\chi$ of $k$-algebras used in the Scheja--Storch construction. However, we will phrase this more categorically than in \cite{SchejaStorch}, as it will benefit us later.

\begin{proposition}\label{prop:naturality-chi-SS}
Consider two endofunctors $F,G:\Alg^\fg_k \to \Alg^\fg_k$ on the category of finitely generated $k$-algebras, where $F(A)=A \otimes_k A$ and $G(A)=\Hom_k \left( \Hom_k \left( A,k \right),A  \right)$. Then there is a natural isomorphism $\chi:F\to G$ whose component at a $k$-algebra $A$ is
\begin{align*}
    \chi_A : A \otimes_k A &\to \Hom_k \left( \Hom_k \left( A,k \right),A \right) \\
    b \otimes c &\mapsto \left[ \varphi \mapsto \varphi(b)c \right].
\end{align*}
\end{proposition}
\begin{proof} This canonical isomorphism is given in \cite[p.181]{SchejaStorch}, so it will suffice for us to verify naturality. Let $g: A \to B$ be any morphism of $k$-algebras. Consider the induced maps $g \otimes g : A \otimes_k A \to B \otimes_k B$ and
\begin{align*}
    g_\ast : \Hom_k \left( \Hom_k(A,k),A \right) &\to \Hom_k \left( \Hom_k \left( B,k \right),B \right) \\
    \psi &\mapsto \left[ \epsilon \mapsto g\circ \psi(\epsilon\circ g) \right].
\end{align*}
It remains to show that the following diagram commutes.
\[ \begin{tikzcd}
    A \otimes_k A \rar["\chi_A" above]\dar["g \otimes g" left] & \Hom_k \left( \Hom_k \left( A,k \right) ,A\right)\dar["g_\ast" right]\\
    B \otimes_k B\rar["\chi_B" below] & \Hom_k \left( \Hom_k \left( B,k \right),B \right)
\end{tikzcd}\]
To see this, we compute $g_*\circ\chi_A=\left[b\otimes c\mapsto[\epsilon\mapsto g((\epsilon\circ g)(b)\cdot c)]\right]$. Note that $\epsilon\circ g:B\to k$, so $(\epsilon\circ g)(b)\in k$. Since $g$ is $k$-linear, we have $g((\epsilon\circ g)(b)\cdot c)=\epsilon(g(b))\cdot g(c)$. Next, we compute $\chi_B\circ(g\otimes g)=\left[b\otimes c\mapsto[\epsilon\mapsto\epsilon(g(b))\cdot g(c)]\right]$. Thus $g_*\circ\chi_A=\chi_B\circ(g\otimes g)$, so the diagram commutes.
\end{proof}

We now let $\Theta := \chi_Q(\Delta)$ denote the image of $\Delta$ under the component of this natural isomorphism at the global algebra $Q$. We have that $\Theta$ is a $k$-linear map $\Theta : \Hom_k(Q,k) \to Q$. Letting $\eta$ denote $\Theta^{-1}(1)$, we obtain a well-defined linear form $\eta: Q \to k$ by~\cite[3.3 Satz]{SchejaStorch}.

\begin{definition} We refer to $\Phi_\eta : Q \times Q \to k$ as the \textit{global Scheja--Storch bilinear form}.
\end{definition}

The B\'ezoutian gives us an explicit formula for $\Delta$. As a result, the global Scheja--Storch form agrees with the B\'ezoutian form.

\begin{proposition}\label{prop:delta-equals-Bezoutian}
In $Q\otimes_k Q$, we have $\Delta=\Bez(f_1,\ldots,f_n)$.
\end{proposition}
\begin{proof}
We first compute
\begin{align*}
    \sum_{i=1}^n \Delta_{ji} \left( X_i - Y_i \right) &= \sum_{i=1}^n \frac{f_j(Y_1, \ldots, Y_{i-1},X_i, \ldots, X_n) - f_j(Y_1, \ldots, Y_i,X_{i+1}, \ldots, X_n)}{(X_i - Y_i)}\cdot(X_i - Y_i) \\
    &= \sum_{i=1}^n f_j(Y_1, \ldots, Y_{i-1},X_i, \ldots, X_n) - f_j(Y_1, \ldots, Y_i,X_{i+1}, \ldots, X_n) \\
    &= f_j(X_1, \ldots, X_n) - f_j(Y_1, \ldots, Y_n).
\end{align*}

Let $\varphi:k\langle X\rangle\otimes_k k\langle X\rangle\xto{\sim} k\langle X,Y\rangle$ be the ring isomorphism given by $\varphi(b\otimes c)=b(X)c(Y)$. The inverse of $\varphi$ is characterized by $\varphi^{-1}(X_i)=x_i\otimes 1$ and $\varphi^{-1}(Y_i)=1\otimes x_i$. It follows that
\begin{align*}
    f_j\otimes 1-1\otimes f_j&=\varphi^{-1}(f_j(X)-f_j(Y)) =\sum_{i=1}^n\varphi^{-1}(\Delta_{ji}(X_i-Y_i))\\
    &=\sum_{i=1}^n\varphi^{-1}(\Delta_{ji})(x_i\otimes 1-1\otimes x_i).
\end{align*}
We may thus set $a_{ij}=\varphi^{-1}(\Delta_{ji})$, and~\cite[3.1 Satz]{SchejaStorch} implies that $\Delta=(\rho\otimes\rho)(\det(a_{ij}))$. On the other hand, $(\rho\otimes\rho)(\varphi^{-1}(\det(\Delta_{ji})))=\Bez(f_1,\ldots,f_n)$ by Definition~\ref{def:Bezoutian}.
\end{proof}

\begin{lemma}\label{lem:Bezoutian-equals-SS} The B\'{e}zoutian bilinear form and the global Scheja--Storch bilinear form are identical.
\end{lemma}
\begin{proof} 
We showed in Proposition~\ref{prop:delta-equals-Bezoutian} that $\Delta$ is the B\'{e}zoutian in $Q \otimes_k Q$. We now show that the associated forms are identical. Pick bases $\{a_i\}$ and $\{b_i\}$ of $Q$ such that
\begin{align*}
    \Delta = \Bez(f_1, \ldots, f_n) = \sum_{i=1}^m a_i \otimes b_i.
\end{align*}
Since the natural isomorphism $\chi$ has $k$-linear components, $\Delta$ is mapped to
\begin{align*}
    \Theta := \chi_Q(\Delta) = \left[ \varphi \mapsto \sum_{i=1}^m \varphi(a_i)b_i \right].
\end{align*}

Thus $\eta:=\Theta^{-1}(1)$ is the linear form $\eta : Q \to k$ satisfying $\sum_{i=1}^m \eta(a_i)b_i = 1$. By Proposition~\ref{prop:equivalent-conditions-bases-dual}, this implies that $\eta$ is the form for which $\{a_i\}$ and $\{b_i\}$ are dual bases. As in Definition~\ref{def:Bezoutian-bilinear-form}, this tells us that $\eta$ is the linear form producing the B\'ezoutian bilinear form.
\end{proof}

\subsection{Local decomposition}\label{sec:local-to-global}

While our discussion of the Scheja--Storch form in the previous section was global, it is perfectly valid to localize at a maximal ideal and repeat the story again \cite[p.180--181]{SchejaStorch}. The fact that $Q$ is an Artinian ring then gives a convenient way to relate the global version of $\eta$ to the local version of $\eta$. This local decomposition has been utilized previously, for example in~\cite{KW-EKL}.

Let $\mathfrak{m}$ be a maximal ideal in $k[x_1, \ldots, x_n]$ at which the morphism $f = (f_1, \ldots, f_n)$ has an isolated root. Letting $\rho_\mathfrak{m}$ denote the quotient map $k \left\langle X \right\rangle_\mathfrak{m} \to Q_\mathfrak{m}$, we have a commutative diagram
\[ \begin{tikzcd}
    k \left\langle X \right\rangle_\mathfrak{m} \otimes_k k \left\langle X \right\rangle_\mathfrak{m}\dar["\rho_\mathfrak{m} \otimes \rho_\mathfrak{m}" left]\rar["\mu_1'"] & k \left\langle X \right\rangle_\mathfrak{m}\dar["\rho_\mathfrak{m}" right]\\
    Q_\mathfrak{m} \otimes_k Q_\mathfrak{m}\rar["\mu'" below] & Q_\mathfrak{m}.
\end{tikzcd} \]

In $k\langle X\rangle_\mathfrak{m}\otimes_k k\langle X\rangle_\mathfrak{m}$, we can again write
\[f_j\otimes 1-1\otimes f_j=\sum_{i=1}^n\tilde{a}_{ij}(X_i\otimes 1-1\otimes X_i)\]
to obtain the local B\'ezoutian $\Delta_\mathfrak{m}:=(\rho_\mathfrak{m}\otimes\rho_\mathfrak{m})(\det(\tilde{a}_{ij})) \in Q_\mathfrak{m} \otimes_k Q_\mathfrak{m}$. Let $\lambda_\mathfrak{m}:Q\to Q_\mathfrak{m}$ be the localization map. From \cite[p.181]{SchejaStorch} we have $\left( \lambda_\mathfrak{m} \otimes \lambda_\mathfrak{m} \right) \left( \Delta \right) = \Delta_\mathfrak{m}$. Via the natural isomorphism $\chi$ in Proposition~\ref{prop:naturality-chi-SS}, we have a commutative diagram of the form
\[ \begin{tikzcd}
    Q \otimes_k Q\dar["\lambda_\mathfrak{m} \otimes \lambda_\mathfrak{m}" left]\rar["\chi_Q" above] & \Hom_k \left( \Hom_k(Q,k),Q \right)\dar["\lambda_{\mathfrak{m}*}"]\\
    Q_\mathfrak{m} \otimes_k Q_\mathfrak{m}\rar["\chi_{Q_\mathfrak{m}}" below] & \Hom_k \left( \Hom_k(Q_\mathfrak{m},k),Q_\mathfrak{m} \right).
\end{tikzcd}\]
Tracing $\Delta$ through this diagram, we see that
\[\begin{tikzcd}[row sep=large]
    \Delta\dar[maps to]\rar[maps to] & \Theta\dar[maps to]\\
    \Delta_\mathfrak{m}\rar[maps to] & \Theta_\mathfrak{m},
\end{tikzcd} \]
where $\Theta_\mathfrak{m} = \chi_{Q_\mathfrak{m}}(\Delta_\mathfrak{m})$. Unwinding $\Theta_\mathfrak{m}=\lambda_{\mathfrak{m}*}(\Theta)$, we find that $\Theta_\mathfrak{m}$ is the map
\begin{align*}
    \Theta_\mathfrak{m} : \Hom_k(Q_\mathfrak{m},k) &\to Q_\mathfrak{m} \\
    \psi &\mapsto \lambda_\mathfrak{m} \circ \Theta \left( \psi \circ \lambda_\mathfrak{m} \right).
\end{align*}

Recall that as $Q$ is a zero-dimensional Noetherian commutative $k$-algebra, the localization maps induce a $k$-algebra isomorphism\footnote{$Q$ is Artinian by~\cite[\href{https://stacks.math.columbia.edu/tag/00KH}{Lemma 00KH}]{stacks}, so the claimed isomorphism exists by~\cite[\href{https://stacks.math.columbia.edu/tag/00JA}{Lemma 00JA}]{stacks}.}
\begin{align*}
    \left( \lambda_\mathfrak{m} \right)_\mathfrak{m} : Q \xto{\sim} \prod_\mathfrak{m} Q_\mathfrak{m}.
\end{align*}
This is reflected by an internal decomposition of $Q$ in terms of orthogonal idempotents~\cite[2.13]{BCRS}, which we now describe (see also~\cite[\href{https://stacks.math.columbia.edu/tag/00JA}{Lemma 00JA}]{stacks}). By the Chinese remainder theorem, we may pick a collection of pairwise orthogonal idempotents $\{e_\mathfrak{m}\}_\mathfrak{m}$ such that $\sum_\mathfrak{m}e_\mathfrak{m}=1$. The internal decomposition of $Q$ is then
\[Q=\bigoplus_\mathfrak{m}Q\cdot e_\mathfrak{m},\]
and the localization maps restrict to isomorphisms $\left.\lambda_\mathfrak{m} \right|_{Q\cdot e_\mathfrak{m}}:Q\cdot e_\mathfrak{m}\xto{\sim} Q_\mathfrak{m}$ with $\lambda_\mathfrak{m}(e_\mathfrak{m})=1$. Moreover, $\lambda_\mathfrak{m}(Q\cdot e_\mathfrak{n})=0$ for any $\mathfrak{n}\neq\mathfrak{m}$.

\begin{proposition}\label{prop:theta-of-linear-form-factoring-lands-in-Qem} Suppose $\ell : Q \to k$ is a linear form which factors through the localization $\lambda_\mathfrak{m}:Q \to Q_\mathfrak{m}$ for some maximal ideal $\mathfrak{m}$. Then $\Theta(\ell)$ lies in $Q\cdot e_\mathfrak{m}$.
\end{proposition}
\begin{proof} 
Recall that $\left.\lambda_\mathfrak{m}\right|_{Q\cdot e_\mathfrak{n}}=0$ for $\mathfrak{n}\neq\mathfrak{m}$. Since $e_\mathfrak{m}\cdot e_\mathfrak{n}=0$ for $\mathfrak{n}\neq\mathfrak{m}$ and $e_\mathfrak{m}$ is idempotent, the localization $\lambda_\mathfrak{m} : Q \to Q_\mathfrak{m}$ can be written as the following composition:
\begin{align*}
    \lambda_\mathfrak{m}:Q \xto{-\cdot e_\mathfrak{m}} Q \xto{\lambda_\mathfrak{m}} Q_\mathfrak{m}.
\end{align*}
Since $\ell$ factors through the localization, it can be written as a composite
\begin{align*}
    \ell:Q \xto{-\cdot e_\mathfrak{m}} Q \xto{\lambda_\mathfrak{m}} Q_\mathfrak{m} \xto{\ell_\mathfrak{m}} k.
\end{align*}
Thus $\Theta(\ell) = \Theta(\ell_\mathfrak{m} \circ  \lambda_\mathfrak{m} \circ \left( e_\mathfrak{m}\cdot- \right))$. Scheja--Storch proved that $\Theta$ respects the $Q$-module structure on $\Hom_k(Q,k)$ given by $a\cdot\sigma=\sigma(a\cdot -)$~\cite[3.3 Satz]{SchejaStorch}. That is, $\Theta(\sigma(a\cdot -))=\Theta(a\cdot\sigma)=a\Theta(\sigma)$ for any $a\in Q$ and $\sigma\in\Hom_k(Q,k)$. Thus
\begin{align*}
    \Theta(\ell) = e_\mathfrak{m} \cdot \Theta \left( \ell_\mathfrak{m} \circ  \lambda_\mathfrak{m}\right),
\end{align*}
so $\Theta(\ell)\in Q\cdot e_\mathfrak{m}$.
\end{proof}

Returning to the Scheja--Storch form, we have the following commutative diagram relating $\Theta_\mathfrak{m}$ and $\Theta$:
\[ \begin{tikzcd}
    \Hom_k(Q,k)\rar["\Theta"] & Q\dar["\lambda_\mathfrak{m}" right]\\
    \Hom_k(Q_\mathfrak{m},k)\uar["-\circ \lambda_\mathfrak{m}" left]\rar["\Theta_\mathfrak{m}" below] & Q_\mathfrak{m}.
\end{tikzcd} \]
This coherence between $\Theta$ and $\Theta_\mathfrak{m}$ allows us to relate the local linear forms $\eta_\mathfrak{m}:=\Theta_\mathfrak{m}^{-1}(1)$ to the global linear form $\eta:=\Theta^{-1}(1)$ in the following way.

\begin{prop}\label{prop:local-to-global-eta}
For each maximal ideal $\mathfrak{m}$ of $Q$, let $\eta_\mathfrak{m}:=\Theta_\mathfrak{m}^{-1}(1):Q_\mathfrak{m}\to k$, and let $\eta:=\Theta^{-1}(1):Q\to k$. Then $\eta=\sum_\mathfrak{m}\eta_\mathfrak{m}\circ\lambda_\mathfrak{m}$.
\end{prop}
\begin{proof}
It suffices to show that $\Theta(\sum_\mathfrak{m}\eta_\mathfrak{m}\circ\lambda_\mathfrak{m})=1$. Since $\eta_\mathfrak{m}=\Theta^{-1}_\mathfrak{m}(1)$ by definition, we have $1=\Theta_\mathfrak{m}(\eta_\mathfrak{m}):=\lambda_\mathfrak{m}(\Theta(\eta_\mathfrak{m}\circ\lambda_\mathfrak{m}))$. By Proposition~\ref{prop:theta-of-linear-form-factoring-lands-in-Qem}, we have $\Theta(\eta_\mathfrak{m}\circ \lambda_\mathfrak{m})\in Q\cdot e_\mathfrak{m}$. Since $\lambda_\mathfrak{m}(\Theta(\eta_\mathfrak{m}\circ\lambda_\mathfrak{m})) = 1$ and $\lambda_\mathfrak{m}|_{Q\cdot e_\mathfrak{m}}$ is an isomorphism sending $e_\mathfrak{m}$ to 1, it follows that $\Theta(\eta_\mathfrak{m}\circ \lambda_\mathfrak{m}) = e_\mathfrak{m}$. Finally, since $\Theta$ is $k$-linear, we have
\begin{align*}
    \Theta\left(\sum_\mathfrak{m}\eta_\mathfrak{m}\circ\lambda_\mathfrak{m}\right)&=\sum_\mathfrak{m}\Theta(\eta_\mathfrak{m}\circ\lambda_\mathfrak{m})\\
    &=\sum_\mathfrak{m}e_\mathfrak{m}=1.\qedhere
\end{align*}
\end{proof}

Using this local decomposition procedure for the linear forms $\eta_\mathfrak{m}$ and $\eta$, we obtain a local decomposition for Scheja--Storch bilinear forms.

\begin{lemma}\label{lem:LGP-for-SS-forms} \textit{(Local decomposition of Scheja--Storch forms)} 
Let $\eta$ and $\eta_\mathfrak{m}$ be as in Proposition~\ref{prop:local-to-global-eta}. Then $\Phi_\eta = \bigoplus_\mathfrak{m} \Phi_{\eta_\mathfrak{m}}$. In particular, the global Scheja--Storch form is a sum over local Scheja--Storch forms
\begin{align*}
    \text{SS}(f) = \sum_\mathfrak{m} \text{SS}_\mathfrak{m}(f).
\end{align*}
\end{lemma}
\begin{proof} 
For each maximal ideal $\mathfrak{m}$, let $\left\{ w_{\mathfrak{m},i} \right\}_i$ be a $k$-vector space basis for $Q_\mathfrak{m}$. Let $\left\{v_{\mathfrak{m},i}\right\}_{\mathfrak{m},i}$ (ranging over all $i$ and all maximal ideals) be a basis of $Q$ such that $\lambda_{\mathfrak{m}} \left( v_{\mathfrak{m},i} \right) = w_{\mathfrak{m},i}$ for each $i$ and $\mathfrak{m}$. We now compare the Gram matrix for $\eta : Q \to k$ and the Gram matrices for $\eta_\mathfrak{m} : Q_\mathfrak{m} \to k$ in these bases. Via the internal decomposition consisting of pairwise orthogonal idempotents, we have $v_{\mathfrak{m},i}\cdot v_{\mathfrak{n},j} = 0$ if $\mathfrak{m} \ne \mathfrak{n}$. Thus
\begin{align*}
    \eta(v_{\mathfrak{m},i}\cdot v_{\mathfrak{n},j}) = 0,
\end{align*}
so the Gram matrix for $\Phi_\eta$ will be a block sum indexed over the maximal ideals. If $\mathfrak{m} = \mathfrak{n}$, then Proposition~\ref{prop:local-to-global-eta} implies
\begin{align*}
    \eta(v_{\mathfrak{m},i} \cdot v_{\mathfrak{m},j}) &= \sum_\mathfrak{n} \eta_\mathfrak{n} \left( \lambda_\mathfrak{n} \left( v_{\mathfrak{m},i} \cdot v_{\mathfrak{m},j} \right) \right) = \eta_\mathfrak{m}(\lambda_\mathfrak{m}\left(v_{\mathfrak{m},i} \cdot v_{\mathfrak{m},j})\right)\\
    &= \eta_\mathfrak{m}\left(w_{\mathfrak{m},i} \cdot w_{\mathfrak{m},j}\right).
\end{align*}
Thus the Gram matrices of $\Phi_\eta$ and $\bigoplus_\mathfrak{m} \Phi_{\eta_\mathfrak{m}}$ are equal, so $\Phi_\eta=\bigoplus_\mathfrak{m}\Phi_{\eta_\mathfrak{m}}$.
\end{proof}

\begin{remark}\label{rem:local-version}
The local Scheja--Storch bilinear form is given by $\Phi_{\eta_\mathfrak{m}}:Q_\mathfrak{m}\times Q_\mathfrak{m}\to k$. Given a basis $\{a_1,\ldots,a_m\}$ of $Q_\mathfrak{m}$, we may write $\Delta_\mathfrak{m}=\sum a_i\otimes b_i$ and define the local B\'ezoutian bilinear form as a suitable dualizing form. Replacing $Q$, $\Delta$, $\Theta$, and $\eta$ with $Q_\mathfrak{m}$, $\Delta_\mathfrak{m}$, $\Theta_\mathfrak{m}$, and $\eta_\mathfrak{m}$, the results of Sections~\ref{sec:Bezoutians} and~\ref{sec:Scheja-Storch} also hold for local B\'ezoutians and the local Scheja--Storch form. In particular, the local analog of Lemma~\ref{lem:Bezoutian-equals-SS} implies that the local Scheja--Storch form is equal to the local B\'ezoutian form.
\end{remark}

\section{Proof of Theorem~\ref{thm:main-thm}}\label{sec:proof-of-main-theorem}
We now relate the Scheja--Storch form to the $\A^1$-degree. The following theorem was first proven in the case where $p$ is a rational zero by Kass and Wickelgren \cite{KW-EKL}, and then in the case where $p$ has finite separable residue field over the ground field in \cite[Corollary~1.4]{trace-paper}. Recent work of Bachmann and Wickelgren~\cite{BW} gives a general result about the relation between local $\A^1$-degrees and Scheja--Storch forms.

\begin{theorem}\label{thm:deg-is-SS} 
Let $\Char{k}\neq 2$. Let $f: \A^n_k \to \A^n_k$ be an endomorphism of affine space with an isolated zero at a closed point $p$. Then we have that the local $\A^1$-degree of $f$ at $p$ and the Scheja--Storch form of $f$ at $p$ coincide as elements of $\GW(k)$:
\begin{align*}
    \deg_p^{\A^1}(f) = \text{SS}_p(f).
\end{align*}
\end{theorem}
\begin{proof} We may rewrite $f$ as a section of the trivial rank $n$ bundle over affine space $\mathcal{O}^n_{\A^n_k} \to \A^n_k$. Under the hypothesis that $p$ is isolated, we may find a neighborhood $X \subseteq \A^n_k$ of $p$ where the section $f$ is non-degenerate (meaning it is cut out by a regular sequence). By \cite[Corollary~8.2]{BW}, the local index of $f$ at $p$ with the trivial orientation, corresponding to the representable Hermitian $K$-theory spectrum $\KO$, agrees with the local Scheja--Storch form as elements of $\KO^0(k)$:
\begin{equation}\label{eqn:ind-is-SS}
\begin{aligned}
    \ind_p(f, \rho_\text{triv}, \KO) = \text{SS}_p(f).
\end{aligned}
\end{equation}
Let $\mathbb{S}$ denote the sphere spectrum in the stable motivic homotopy category $\mathcal{SH}(k)$. It is a well-known fact that Hermitian $K$-theory receives a map from the sphere spectrum, inducing an isomorphism $\pi_0(\mathbb{S}) \xto{\sim} \pi_0(\KO)$ if $\Char{k}\neq 2$ (see for example \cite[6.9]{Hornbostel} for more detail); this is the only place where we use the assumption that $\Char{k}\neq 2$. Combining this with the fact that that $\pi_0(\mathbb{S}) = \GW(k)$ under Morel's degree isomorphism, we observe that Equation~\ref{eqn:ind-is-SS} is really an equality in $\GW(k)$. By \cite[Theorem~7.6,\ Example~7.7]{BW}, the local index associated to the representable theory agrees with the local $\A^1$-degree:
\begin{align*}
    \ind_p(f, \rho_\text{triv},\KO) = \deg_p^{\A^1}(f).
\end{align*}
Combining these equalities gives the desired equality in $\GW(k)$.
\end{proof}

\begin{remark}
Bachmann and Wickelgren in fact show that $\deg_Z^{\A^1}(f)=\text{SS}_Z(f)$ for any isolated zero locus $Z$ of $f$ \cite[Corollary 8.2]{BW}. This gives an alternate viewpoint on the local decomposition described in Lemma~\ref{lem:LGP-for-SS-forms}
\end{remark}

\begin{corollary}\label{cor:local-a1-deg-is-bezoutian} 
Let $\Char{k}\neq 2$. The local B\'{e}zoutian bilinear form is the local $\A^1$-degree.
\end{corollary}
\begin{proof}
As discussed in Remark~\ref{rem:local-version}, we can modify Lemma~\ref{lem:Bezoutian-equals-SS} to the local case by replacing $Q$, $\Delta$, $\Theta$, and $\eta$ with $Q_\mathfrak{m}$, $\Delta_\mathfrak{m}$, $\Theta_\mathfrak{m}$, and $\eta_\mathfrak{m}$. The local B\'ezoutian form is thus equal to the local Scheja--Storch form, which is equal to the local $\A^1$-degree by Theorem~\ref{thm:deg-is-SS}.
\end{proof}

In contrast to previous techniques for computing the local $\A^1$-degree at rational or separable points, Corollary~\ref{cor:local-a1-deg-is-bezoutian} gives an algebraic formula for the local $\A^1$-degree at any closed point.

As a result of the local decomposition of Scheja--Storch forms, the B\'{e}zoutian form agrees with the $\A^1$-degree globally as well.

\begin{corollary}\label{cor:bezoutian-is-global-degree}
Let $\Char{k}\neq 2$. The B\'ezoutian bilinear form is the global $\A^1$-degree.
\end{corollary}
\begin{proof}
Let $\Phi_\eta$ denote the B\'ezoutian bilinear form, which is equal to the global Scheja--Storch bilinear form by Lemma~\ref{lem:Bezoutian-equals-SS}. By Lemma~\ref{lem:LGP-for-SS-forms}, the global Scheja--Storch form decomposes as a block sum of local Scheja--Storch forms. By Theorem~\ref{thm:deg-is-SS}, the local Scheja--Storch bilinear form agrees with the local $\A^1$-degree. Finally, we have that the sum of local $\A^1$-degrees is the global $\A^1$-degree. Putting this all together, we have
\begin{equation}\label{eqn:bez-ss-a1-deg}
\begin{aligned}
    \Phi_\eta = \text{SS}(f) = \sum_\mathfrak{m} \text{SS}_\mathfrak{m}(f) = \sum_\mathfrak{m} \deg_\mathfrak{m}^{\A^1}(f) = \deg^{\A^1}(f).
\end{aligned}
\qedhere
\end{equation}
\end{proof}

\begin{remark}
It is not known if $\GW$ is represented by $\KO$ over fields of characteristic 2, which is the source of our assumption that $\Char{k}\neq 2$. If this problem is resolved, one can remove any characteristic restrictions from our results. Alternately, Lemma~\ref{lem:LGP-for-SS-forms} implies Corollaries~\ref{cor:local-a1-deg-is-bezoutian} and~\ref{cor:bezoutian-is-global-degree} if all roots of $f$ satisfy $\deg_p^{\A^1}(f)=\text{SS}_p(f)$. By \cite{KW-EKL}, \cite{trace-paper}, and \cite[Proposition 34]{KW-cubic}, Corollaries~\ref{cor:local-a1-deg-is-bezoutian} and~\ref{cor:bezoutian-is-global-degree} are true in any characteristic if all roots of $f$ are rational, \'etale, or separable.
\end{remark}

\subsection{Computing the B\'ezoutian bilinear form}\label{sec:computing algorithm}
We now prove Theorem~\ref{thm:main-thm} by describing a method for computing the class in $\GW(k)$ of the B\'ezoutian bilinear form in terms of the B\'ezoutian.

\begin{proof}[Proof of Theorem~\ref{thm:main-thm}] Let $R$ denote either a global algebra $Q$ or a local algebra $Q_\mathfrak{m}$. Let $\{\alpha_i\}$ be any basis for $R$, and express
\begin{align*}
    \Bez(f_1, \ldots, f_n) &= \sum_{i,j} B_{i,j}\alpha_i \otimes \alpha_j.
\end{align*}
Rewriting this, we have
\begin{align*}
    \Bez(f_1, \ldots, f_n) &= \sum_i \alpha_i \otimes \left( \sum_j B_{i,j}\alpha_j \right).
\end{align*}
Let $\beta_i := \sum_j B_{i,j} \alpha_j$, so that $\left\{ \alpha_i \right\}$ and $\left\{ \beta_i \right\}$ are dual bases. Then for any linear form $\lambda: R \to k$ for which $\left\{ \alpha_i \right\}$ and $\left\{ \beta_i \right\}$ are dual, we will have that $\Phi_\lambda$ agrees with the global or local $\A^1$-degree (depending on our choice of $R$) by Corollaries~\ref{cor:local-a1-deg-is-bezoutian} and~\ref{cor:bezoutian-is-global-degree}. Let $\lambda$ be such a form. The product of $\alpha_i$ and $\beta_j$ is given by
\begin{align*}
    \alpha_i \beta_j &= \alpha_i \cdot \sum_s B_{j,s} \alpha_s.
\end{align*}
Applying $\lambda$ to each side, we get an indicator function
\begin{align*}
    \delta_{ij} &= \lambda(\alpha_i \beta_j) = \lambda \left( \alpha_i \sum_s B_{j,s} \alpha_s \right) = \sum_s B_{j,s} \lambda \left( \alpha_i \alpha_s \right).
\end{align*}
Varying over all $i,j,s$, this equation above tells us that the identity matrix is equal to the product of the matrix $(B_{j,s})$ and the matrix $(\lambda(\alpha_i\alpha_s))=(\lambda(\alpha_s\alpha_i))$. Explicitly, we have that
\begin{align*}
    \begin{pmatrix} 1 & 0 & \cdots & 0 \\
    0 & 1 & \cdots & 0 \\
    \vdots & \vdots & \ddots & \vdots \\ 
    0 & 0 & \cdots & 1\end{pmatrix} &=
    \begin{pmatrix} B_{1,1} & B_{1,2} & \cdots & B_{1,m} \\
    B_{2,1} & B_{2,2} & \cdots & B_{2,m} \\
    \vdots & \vdots & \ddots & \vdots \\
    B_{m,1} & B_{m,2} & \cdots & B_{m,m}
    \end{pmatrix}
    \begin{pmatrix} \lambda(\alpha_1^2) & \lambda(\alpha_1 \alpha_2) & \cdots & \lambda(\alpha_1 \alpha_m) \\
    \lambda(\alpha_2 \alpha_1) & \lambda(\alpha_2^2) & \cdots & \lambda(\alpha_2 \alpha_m) \\
    \vdots & \vdots & \ddots & \vdots \\
    \lambda(\alpha_m\alpha_1) & \lambda(\alpha_m \alpha_2) & \cdots & \lambda(\alpha_m^2)
    \end{pmatrix}.
\end{align*}

Thus the Gram matrix for $\Phi_\lambda$ in the basis $\{\alpha_i\}$ is $(B_{i,j})^{-1}$. Since any symmetric bilinear form can be diagonalized, there is an invertible $m\times m$ matrix $S$ such that $S^T\cdot(B_{i,j})\cdot S$ is diagonal. Since
$(S^T\cdot(B_{i,j})\cdot S)\cdot (S^{-1}\cdot (B_{i,j})^{-1}\cdot (S^{-1})^T)$
is equal to the identity matrix, it follows that $S^{-1}\cdot (\lambda(\alpha_i\alpha_j))\cdot (S^{-1})^T$ is diagonal with entries inverse to the diagonal entries of $S^T\cdot(B_{i,j})\cdot S$. Applying the equality $\left\langle a \right\rangle = \left\langle 1/a \right\rangle$ along the diagonals, it follows that $(B_{i,j})^{-1}$ and $(B_{i,j})$ define the same element in $\GW(k)$. Theorem~\ref{thm:main-thm} now follows from Corollaries~\ref{cor:local-a1-deg-is-bezoutian} and~\ref{cor:bezoutian-is-global-degree}.
\end{proof}

The following tables describe algorithms for computing the global and local $\A^1$-degrees in terms of the B\'ezoutian bilinear form. A Sage implementation of these algorithms is available at~\cite{code}.

\begin{figure}[h] 
  \centering
  \textbf{Computing the global} $\A^1$-\textbf{degree via the B\'ezoutian}:
\begin{enumerate}
    \item Compute the $\Delta_{ij}$ and the image of their determinant $\Bez(f) = \det \left( \Delta_{ij} \right)$ in $k[X,Y]/(f(X),f(Y))$.
    \item Pick a $k$-vector space basis $a_1, \ldots, a_m$ of $Q = k[X_1, \ldots, X_n]/(f_1, \ldots, f_n)$. Find $B_{i,j}\in k$ such that
    \begin{align*}
        \Bez(f) &= \sum_{i=1}^m B_{i,j}a_i(X)a_j(Y).
    \end{align*}
    \item The matrix $B= \left( B_{i,j} \right)$ represents $\deg^{\A^1}(f)$. Diagonalize $B$ to write its class in $\GW(k)$.
\end{enumerate}
\label{fig:global-deg}
\end{figure}

\begin{figure}[h] 
  \centering
  \textbf{Computing the local} $\A^1$-\textbf{degree via the B\'ezoutian}:
\begin{enumerate}
    \item Compute the $\Delta_{ij}$ and the image of their determinant $\Bez(f) = \det \left( \Delta_{ij} \right)$ in $k[X,Y]/(f(X),f(Y))$.
    \item Pick a $k$-vector space basis $a_1, \ldots, a_m$ of $Q_\mathfrak{m} = k[X_1, \ldots, X_n]_\mathfrak{m}/(f_1, \ldots, f_n)$. Find $B_{i,j}\in k$ such that
    \begin{align*}
        \Bez(f) &= \sum_{i=1}^m B_{i,j}a_i(X)a_j(Y).
    \end{align*}
    \item The matrix $B= \left( B_{i,j} \right)$ represents $\deg^{\A^1}_\mathfrak{m}(f)$. Diagonalize $B$ to write its class in $\GW(k)$.
\end{enumerate}
\label{fig:local-deg}
\end{figure}

\section{Calculation rules}\label{sec:calculation-rules}
Using the B\'ezoutian characterization of the $\A^1$-degree, we are able to establish various calculation rules for local and global $\A^1$-degrees. See~\cite{kst,quick2020representability} for related results in the local case. 

\begin{proposition}\label{prop:same ideals}
Suppose that $f = (f_1, \ldots, f_n)$ and $g = (g_1, \ldots, g_n)$ are endomorphisms of affine space that generate the same ideal
\begin{align*}
    I = (f_1, \ldots, f_n) = (g_1, \ldots, g_n) \triangleleft k[x_1, \ldots, x_n].
\end{align*}
If $\Bez(f)=\Bez(g)$ in $k[X,Y]$, then $\deg^{\A^1}(f)=\deg^{\A^1}(g)$, and $\deg^{\A^1}_p(f)=\deg^{\A^1}_p(g)$ for all $p$.
\end{proposition}
\begin{proof} 
We may choose the same basis for $Q = k[x_1,\ldots,x_n]/I$ (or $Q_p$ in the local case) in our computation for the degrees of $f$ and $g$. The B\'ezoutians $\Bez(f) = \Bez(g)$ will have the same coefficients in this basis, so their Gram matrices will coincide.
\end{proof}

The following result is the global analogue of \cite[Lemma 14]{quick2020representability}.

\begin{lemma}\label{lemma:composition-linear-transformation}
Let $f=(f_1,\ldots,f_n):\A^n_k\rightarrow \A^n_k$ be an endomorphism of $\A^n_k$ with only isolated zeros. Let $A\in k^{n\times n}$ be an invertible matrix. Then
\[\langle\det A\rangle\cdot\deg^{\A^1}(f)=\deg^{\A^1}(A\circ f)\]
as elements of $\GW(k)$.
\end{lemma}
\begin{proof}
Write $A=(a_{ij})$
and 
\[\Delta_{ij}^g=\frac{g_i(X_1,\ldots,X_j,Y_{j+1},\ldots,Y_n)-g_i(X_1,\ldots,X_{j-1},Y_j\ldots,Y_n)}{X_j-Y_j},\]
where $g$ is either $f$ or $A\circ f$. Then $\Delta_{ij}^{A\circ f}=\sum_{l=1}^n a_{il}\Delta_{lj}^f$, and thus $(\Delta_{ij}^{A\circ f})=A\cdot (\Delta_{ij}^f)$ as matrices over $k[X,Y]$. The ideals generated by $A\circ(f_1,\ldots,f_n)$ and $(f_1,\ldots,f_n)$ are equal, and the images in $Q\otimes_k Q$ of $\det (\Delta_{ij}^{A\circ f})$ and $\det A\cdot \det (\Delta f_{ij})$ are equal. Thus the Gram matrix of the B\'ezoutian bilinear form for $A\circ f$ is $\det A$ times the Gram matrix of the B\'ezoutian bilinear form for $f$. Proposition~\ref{prop:same ideals} then proves the claim.
\end{proof}

\begin{example} We may apply Lemma~\ref{lemma:composition-linear-transformation} in the case where $A$ is a permutation matrix associated to some permutation $\sigma \in \Sigma_n$. Letting $f_\sigma := \left( f_{\sigma(1)}, \ldots, f_{\sigma(n)} \right)$,
we observe that
\begin{align*}
    \deg_p^{\A^1}(f_\sigma) = \langle\sgn(\sigma)\rangle \cdot \deg_p^{\A^1}(f)
\end{align*}
at any isolated zero $p$ of $f$, and an analogous statement is true for global degrees as well.
\end{example}

Next, we prove a generalization of \cite[Lemma 12]{kst}.

\begin{lemma}\label{lemma:unipotent-transformation}
Let $f,g:\A^n_k\rightarrow \A^n_k$ be two endomorphisms of $\A^n_k$ with only isolated zeros. Let $L\in M_n(k[x_1,\ldots,x_n])$ be a unipotent $n\times n$ matrix. Then $\deg^{\A^1}(f\circ g)=\deg^{\A^1}(f\circ L\circ g)$, and $\deg^{\A^1}_p(f\circ g)=\deg^{\A^1}_p(f\circ L\circ g)$ for any isolated zero $p$.
\end{lemma}
\begin{proof}
We first show that because $g$ only has isolated zeros, $L\circ g$ only has isolated zeros as well. Since $L$ is unipotent, there exists some $m$ such that $L^m=I_n$, where $I_n$ is the $n\times n$ identity matrix. Suppose that $p$ is a zero of $L\circ g$. That is, $L(p)\cdot g(p)=0$, where we think of $g=(g_1,\ldots,g_n)$ as a column vector. Then
\begin{align*}
    0&=L^{m-1}(p)\cdot 0 =L^m(p)\cdot g(p) =g(p),
\end{align*}
so $p$ is a zero of $g$. In particular, all zeros of $L\circ g$ are zeros of $g$. As $g$ only has isolated zeros by assumption, it follows that $L\circ g$ only has isolated zeros as well. We remark that because $f$ and $L\circ g$ only have isolated zeros, their composition $f\circ L\circ g$ only has isolated zeros.

Just as in \cite[Lemma 12]{kst},
we now define $L_t=I_n+t\cdot(L-I_n)$. Note that because $L$ is unipotent, $L-I_n$ is nilpotent, so $t\cdot(L-I_n)$ is nilpotent as well. Thus $L_t$ is unipotent, and hence $f\circ L_t\circ g$ only has isolated zeros for all $t$. Set
\[\widetilde{Q}=\frac{k[t][x_1,\ldots,x_n]}{(f\circ L_t\circ g)}.\]
Then~\cite[p. 182]{SchejaStorch} gives us a Scheja--Storch form $\widetilde{\eta}:\widetilde{Q}\to k[t]$ such that the bilinear form $\Phi_{\widetilde{\eta}}:\widetilde{Q}\times\widetilde{Q}\to k[t]$ is symmetric and non-degenerate. By Harder's theorem~\cite[Lemma 30]{KW-EKL}, the stable isomorphism class of $\Phi_{\widetilde{\eta}}\otimes_k k(t_0)\in\GW(k)$ is independent of $t_0\in\A^1_k(k)$. In particular, the Scheja--Storch bilinear forms of $f\circ L_0\circ g=f\circ g$ and $f\circ L_1\circ g=f\circ L\circ g$ are isomorphic. One can repeat the same argument in the local case to show that the local Scheja--Storch bilinear forms of $f\circ L_0\circ g$ and $f\circ L_1\circ g$ are isomorphic as well.
\end{proof}

As an immediate corollary, we get the global analogue of \cite[Lemma 13]{quick2020representability}.

\begin{cor}
Let $f=(f_1,\ldots,f_n):\A^n_k\rightarrow \A^n_k$ be an endomorphism of $\A^n_k$ with only isolated zeros. Let $h\in k[x_1,\ldots,x_n]$, and define
\[f'=(f_1,\ldots,f_s+h\cdot f_t,\ldots,f_n):\A^n_k\rightarrow \A^n_k\]
for some $s\neq t$. Then $\deg^{\A^1}(f)=\deg^{\A^1}(f')$.
\end{cor}
\begin{proof}
Let $\id=(x_1,\ldots,x_n):\A^n_k\to\A^n_k$ be the identity morphism. Let $a_{ii}=1$ for $1\leq i\leq n$, $a_{st}=h$, and $a_{ij}=0$ otherwise. Then $L=(a_{ij})$ is unipotent with $f'=L\circ f$. By Lemma~\ref{lemma:unipotent-transformation}, $f'=\id\circ L\circ f$ and $f=\id\circ f$ have the same global degree.
\end{proof}

The following product rule is a consequence of Morel's proof that the $\A^1$-degree is a ring isomorphism~\cite[Lemma 6.3.8]{Morel-Trieste}. We give a more hands-on proof of this product rule. While we state the product rule for global degrees, the same arguments imply that the product rule holds for local degrees. See~\cite[Theorem 13]{kst} and~\cite[Theorem 26]{quick2020representability} for an analogous proof of the product rule for local degrees at rational points.

\begin{proposition}[Product rule]
\label{prop:product-rule}
Let $f,g:\A^n_k\rightarrow \A^n_k$ be two endomorphisms of $\A^n_k$ with only isolated zeros.
Then $\deg^{\A^1}(f\circ g)=\deg^{\A^1}(f)\cdot\deg^{\A^1}(g)$.
\end{proposition}
\begin{proof}
We follow the proofs of \cite[Theorem 13]{kst} and \cite[Theorem 26]{quick2020representability}. The general idea is to mimic the Eckmann--Hilton argument~\cite{EH62}. Let $x:=(x_1,\ldots,x_n)$ and $y:=(y_1,\ldots,y_n)$. Define $\tilde{f},\tilde{g}:\A^n\times \A^n\rightarrow \A^n\times \A^n$ by $\tilde{f}(x,y)=(f(x),y)$ and $\tilde{g}(x,y)=(g(x),y)$. Since $(f\circ g,y)$ and $\tilde{f}\circ\tilde{g}$ define the same ideal in $k[x,y]$ and have the same B\'ezoutian, we have $\deg^{\A^1}(f\circ g)=\deg^{\A^1}(\tilde{f}\circ\tilde{g})$ by Proposition~\ref{prop:same ideals}.

Let $g\times f:\A^n_k\times\A^n_k\to\A^n_k\times\A^n_k$ be given by $(g\times f)(x,y)=(g(x),f(y))$. Using Lemma \ref{lemma:unipotent-transformation} repeatedly, we will show that $\deg^{\A^1}(\tilde{f}\circ \tilde{g})= \deg^{\A^1}(g\times f)$. Let $I_n$ be the $n\times n$ identity matrix, and let
\[
L_1=\begin{pmatrix}
I_n & 0\\
-I_n & I_n
\end{pmatrix},\quad
L_2=\begin{pmatrix}
I_n & I_n\\
0 &I_n
\end{pmatrix},\quad
A=\begin{pmatrix}
0 & -I_n\\
I_n & 0\end{pmatrix}.\]
Since $L_1$ and $L_2$ are unipotent, Lemma~\ref{lemma:unipotent-transformation} implies that 
\begin{align*}
    \deg^{\A^1}(\tilde{f}\circ\tilde{g})&=\deg^{\A^1}(\tilde{f}\circ L_1\circ\tilde{g})\\
    &=\deg^{\A^1}(\tilde{f}\circ L_2\circ (L_1\circ \tilde{g}))\\
    &=\deg^{\A^1}(\tilde{f}\circ L_1\circ (L_2\circ L_1\circ \tilde{g})).
\end{align*}
One can check that $A\circ\tilde{f}\circ L_1\circ L_2\circ L_1\circ\tilde{g}=g\times f$. By Lemma~\ref{lemma:composition-linear-transformation}, we have 
\begin{align*}
\langle\det A\rangle\cdot\deg^{\A^1}(\tilde{f}\circ\tilde{g})&=\langle\det A\rangle\cdot\deg^{\A^1}(\tilde{f}\circ L_1\circ L_2\circ L_1\circ \tilde{g})\\
&=\deg^{\A^1}(g\times f).
\end{align*}
Since $\det A=1$, it just remains to show that $\deg^{\A^1}(g\times f)=\deg^{\A^1}(g)\cdot\deg^{\A^1}(f)$. Let $a_1,\ldots,a_m$ be a basis for $\frac{k[x_1,\ldots,x_n]}{(g_1,\ldots,g_n)}$ and $a_1',\ldots,a_{m'}'$ be a basis for $\frac{k[y_1,\ldots,y_n]}{(f_1,\ldots,f_n)}$.
Write $\Bez(g)=\sum_{i,j=1}^mB_{ij}a_i\otimes a_j$ and $\Bez(f)=\sum_{i,j=1}^{m'}B'_{ij}a_i'\otimes a_j'$. By Theorem~\ref{thm:main-thm}, $(B_{ij})$ and $(B'_{ij})$ are the Gram matrices for $\deg^{\A^1}(g)$ and $\deg^{\A^1}(f)$, respectively. Next, we have $\Bez(g\times f)=\Bez(g)\cdot \Bez(f)$, since \[(\Delta_{ij}^{g\times f})=\begin{pmatrix}
(\Delta_{ij}^g) & 0 \\ 
0 & (\Delta_{ij}^f)
\end{pmatrix}.\]

Note that $\{a_i(x)a_{i'}'(y)\}_{i,i'=1}^{m,m'}$ is a basis of $\frac{k[x_1,\ldots,x_n,y_1,\ldots,y_n]}{(g_1(x),\ldots,g_n(x),f_1(y),\ldots,f_n(y))}$. In this basis, we have
\begin{align*}
    \Bez(g)\cdot \Bez(f)=\sum_{i,j=1}^m\sum_{i',j'=1}^{m'}B_{ij}B'_{i'j'}a_ia_{i'}'\otimes a_ja_{j'}',
\end{align*}
so the Gram matrix of $\deg^{\A^1}(g\times f)$ is the tensor product $(B_{ij})\otimes (B_{ij}')$. We thus
 we have an equality $\deg^{\A^1}(g\times f)=\deg^{\A^1}(g)\cdot \deg^{\A^1}(f)$ in $\GW(k)$.
\end{proof}

\section{Examples}\label{sec:examples}
We now give a few remarks and examples about computing the B\'ezoutian.

\begin{remark} It is not always the case that the determinant $\det(\Delta_{ij})\in k[X,Y]$ is symmetric. For example, consider the morphism $f: \A^2_k \to \A^2_k$ sending $\left( x_1,x_2 \right)\mapsto \left( x_1x_2, x_1 + x_2 \right)$. Then the B\'{e}zoutian is given by
\begin{align*}
    \Bez(f) &= \det \begin{pmatrix} X_2 & Y_1 \\ 1 & 1 \end{pmatrix}  = X_2 - Y_1.
\end{align*}
However, the B\'{e}zoutian is symmetric once we pass to the quotient $\frac{k[X,Y]}{(f(X),f(Y))}$~\cite[2.12]{BCRS}. Continuing the present example, let $\left\{ 1,x_2 \right\}$ be a basis for the algebra $Q = k[x_1,x_2]/(x_1x_2, x_1 + x_2)$. Then we have that
\begin{align*}
    \Bez(f) = X_2 - Y_1 = X_2 + Y_2, 
\end{align*}
which is symmetric. Moreover, the B\'ezoutian bilinear form is represented by $\left(\begin{smallmatrix} 0 & 1 \\ 1 & 0 \end{smallmatrix}\right)$, so $\deg^{\A^1}(f) = \mathbb{H}$.
\end{remark}

\begin{example} Let $k = \F_p(t)$, where $p$ is an odd prime, and consider the endomorphism of the affine plane given by
\begin{align*}
    f:\Spec \F_p(t)[x_1,x_2] &\to \Spec \F_p(t)[x_1,x_2] \\
    (x_1,x_2) &\mapsto \left( x_1^p - t, x_1x_2 \right).
\end{align*}
As the residue field of the zero of $f$ is not separable over $k$, existing strategies for computing the local $\A^1$-degree are insufficient. Our results allow us to compute this $\A^1$-degree. The B\'{e}zoutian is given by
\begin{align*}
    \Bez(f) &= \det \begin{pmatrix} \frac{X_1^p - Y_1^p}{X_1 - Y_1} & 0 \\ X_2 & Y_1 \end{pmatrix} \\
    &= X_1^{p-1}Y_1 + X_1^{p-2} Y_1^2 + \ldots + X_1Y_1^{p-1}+Y_1^p\\
    &=X_1^{p-1}Y_1 + X_1^{p-2} Y_1^2 + \ldots + X_1Y_1^{p-1}+t.
\end{align*}
In the basis $\{1,x_1,\ldots,x_1^{p-1}\}$ of $Q$, the B\'ezoutian bilinear form consists of a $t$ in the upper left corner and a 1 in each entry just below the anti-diagonal. Thus
\begin{align*}
    \deg^{\A^1}(f) = \deg^{\A^1}_{(t^{1/p},0)} (f) = \langle t\rangle+\frac{p-1}{2}\mathbb{H}.
\end{align*}
\end{example}

\begin{example}
Let $f_1=(x_1-1)x_1x_2$ and $f_2=(ax_1^2-bx_2^2)$ for some $a,b\in k^\times$ with $\tfrac{a}{b}$ not a square in $k$. Then $f=(f_1,f_2)$ has isolated zeros at $\mathfrak{m}:=(x_1-0,x_2-0)$ and $\mathfrak{n}:=(x_1-1,x_2^2-\tfrac{a}{b})$. We will use B\'ezoutians to compute the local degrees $\deg_\mathfrak{m}^{\A^1}(f)$ and $\deg_\mathfrak{n}^{\A^1}(f)$, as well as the global degree $\deg^{\A^1}(f)$. Let
\[Q = \frac{k[x_1,x_2]}{((x_1-1)x_1x_2, ax_1^2 - bx_2^2)}.\]

We first compute the global B\'{e}zoutian as
\begin{align*}
    \Bez(f)=&\det\begin{pmatrix}(X_1+Y_1-1)X_2&a(X_1+Y_1)\\ Y_1^2-Y_1&-b(X_2+Y_2)\end{pmatrix}\\
    =& -a(X_1Y_1^2-X_1Y_1+Y_1^3-Y_1^2)\\
    &-b(X_1X_2^2+X_2^2Y_1-X_2^2+X_1X_2Y_2+X_2Y_1Y_2-X_2Y_2).
\end{align*}
In the basis $\{ 1,x_1,x_2,x_1^2,x_1x_2, x_1^2\}$ of $Q$, the B\'ezoutian is given by
\begin{align*}
    \Bez(f) =& -a \left( X_1 Y_1^2 - X_1 Y_1 + Y_1^3 - Y_1^2 + X_1^3 +  X_1^2 Y_1 - X_1^2 \right) \\
    &- b ( X_1 X_2 Y_2 + X_2 Y_1 Y_2 - X_2 Y_2).
\end{align*}
We now write the B\'ezoutian matrix given by the coefficients of $\Bez(f)$.
\begin{center}
    \begin{tabular}{r | c c c c c c}
    & $1$ & $X_1$ & $X_2$ & $X_1^2$ & $X_1 X_2$ & $X_1^3$ \\
    \hline
    1 & 0 & 0 & 0 & $a$ & 0 & $-a$\\
    $Y_1$ & 0 & $a$ & 0 & $-a$ & 0 & 0\\
    $Y_2$ & 0 & 0 & $b$ & 0 & $-b$ & 0\\
    $Y_1^2$ & $a$ & $-a$ & 0 & 0 & 0 & 0\\
    $Y_1 Y_2$ & 0 & 0 & $-b$ & 0 & 0 & 0\\
    $Y_1^3$ & $-a$ & 0 & 0 & 0 & 0 & 0
    \end{tabular}
\end{center}
One may check (e.g. with a computer) that this is equal to $3\mathbb{H}$ in $\GW(k)$.

In $Q_\mathfrak{m}$, we have that $x_1^2 x_2 = x_1 x_2 = 0$ and $x_1^3 = \frac{b}{a}x_1 x_2^2 = 0$. In the basis $\{1,x_1,x_2,x_1^2\}$ of $Q_\mathfrak{m}$, the global B\'ezoutian reduces to
\begin{align*}
    \Bez(f) &= -a \left( X_1 Y_1^2 - X_1 Y_1 - Y_1^2 + X_1^2 Y_1 - X_1^2 \right) + bX_2 Y_2
\end{align*}
We thus get the B\'{e}zoutian matrix at $\mathfrak{m}$.
\begin{center}
    \begin{tabular}{r | c c c c}
    & $1$ & $X_1$ & $X_2$ & $X_1^2$ \\
    \hline
    1 & 0 & 0 & 0 & $a$ \\ 
    $Y_1$ & 0 & $a$ & 0 & $-a$ \\
    $Y_2$ & 0 & 0 & $b$ & 0 \\
    $Y_1^2$ & $a$ & $-a$ & 0 & 0 \\
    \end{tabular}
\end{center}
This is $\mathbb{H} + \langle a,b\rangle$ in $\GW(k)$.

In $Q_\mathfrak{n}$, we have $x_1 = 1$. In the basis $\{1,x_2\}$ for $Q_\mathfrak{n}$, the B\'ezoutian reduces to
\begin{align*}
    \Bez(f) &= -a -bX_2 Y_2.
\end{align*}
We can then write the B\'ezoutian matrix at $\mathfrak{n}$.
\begin{center}
    \begin{tabular}{r | c c}
    & $1$ & $X_2$ \\
    \hline
    $1$ & $-a$ & 0 \\
    $Y_2$ & 0 & $-b$
    \end{tabular}
\end{center}
This is $\langle -a, -b\rangle$ in $\GW(k)$. Note that $\langle -a,-b\rangle$ need not be equal to $\mathbb{H}$. However, this does not contradict \cite[Theorem 2]{quick2020representability}, since $\mathfrak{n}$ is a non-rational point. 

Putting these computations together, we see that
\begin{align*}
    \deg_\mathfrak{m}^{\A^1}(f) + \deg_\mathfrak{n}^{\A^1}(f) &= \mathbb{H} + \left\langle a,b \right\rangle + \left\langle -a,-b \right\rangle = 3 \mathbb{H} = \deg^{\A^1}(f).
\end{align*}
\end{example}

\section{Application: the $\A^1$-Euler characteristic of Grassmannians}\label{sec:Gr}
As an application of Theorem~\ref{thm:main-thm}, we compute the $\A^1$-Euler characteristic of various low-dimensional Grassmannians in Example~\ref{ex:grassmannian} and Figure~\ref{fig:Euler-char-grassmannians}. These computations suggest a recursive formula for the $\A^1$-Euler characteristic of an arbitrary Grassmannian, which we prove in Theorem~\ref{thm:grassmannian}. This formula is analogous to the recursive formulas for the Euler characteristics of complex and real Grassmannians. Theorem~\ref{thm:grassmannian} is probably well-known, and the proof is essentially a combination of results of Hoyois, Levine, and Bachmann--Wickelgren.

\subsection{The $\A^1$-Euler characteristic}
Let $X$ be a smooth, proper $k$-variety of dimension $n$ with structure map $\pi:X\rightarrow \Spec k$. Let $p:T_X\rightarrow X$ denote the tangent bundle of $X$.
The $\A^1$-Euler characteristic $\chi^{\A^1}(X)\in\GW(k)$ is a refinement of the classical Euler characteristic. In particular, if $k=\mathbb{R}$, then $\operatorname{rank}\chi^{\A^1}(X)=\chi(X(\mathbb{C}))$ and $\operatorname{sgn}\chi^{\A^1}(X)=\chi(X(\mathbb{R}))$.
There exist several equivalent definitions of the $\A^1$-Euler characteristic \cite{levine2020aspects,LR,ABOWZ}. 
For example, we may define $\chi^{\A^1}(X)$ to be the $\pi$-pushforward of the $\A^1$-Euler class 
\[e(T_X):=z^*z_*1_X\in \widetilde{\operatorname{CH}}^n(X,\omega_{X/k}),\]
of the tangent bundle \cite{levine2020aspects}, where $z:X\to T_X$ is the zero section and $\widetilde{\operatorname{CH}}^d(X,\omega_{X/k})$ is the Chow--Witt group defined by Barge--Morel \cite{BM,fasel}. That is,
\[\chi^{\A^1}(X):=\pi_*(e(T_X))\in\widetilde{\operatorname{CH}}^0(\Spec k)=\GW(k).\]

Analogous to the classical case \cite{MilnorDiffViewpoint}, the $\A^1$-Euler characteristic can be computed as the sum of local $\A^1$-degrees at the zeros of a general section of the tangent bundle using the work of Kass--Wickelgren \cite{BW,KW-cubic,levine2020aspects}. We now describe this process.
Let $\sigma$ be a section of $T_X$ which only has isolated zeros. For a zero $x$ of $\sigma$, choose Nisnevich coordinates\footnote{\textit{Nisnevich coordinates} consist of an open neighborhood $U$ of $x$ and an \'etale map $\psi:U\to\A^n_k$ that induces an isomorphism of residue fields $k(x)\cong k(\psi(x))$ \cite[Definition 18]{KW-cubic}.} $\psi:U\to\A^n_k$ around $x$. Since $\psi$ is \'{e}tale, it induces an isomorphism of tangent spaces and thus yields local coordinates around $x$. Shrinking $U$ if necessary, we can trivialize $T_X\vert_U\cong U\times \A^n_k$.
The chosen Nisnevich coordinates $(\psi,U)$ and trivialization $\tau:T_X\vert_U\cong U\times\A^n_k$ each define distinguished elements $d_\psi,d_\tau\in\det T_X\vert_U$. In turn, this yields a distinguished section $d$ of $\sheafHom(\det T_X\vert_U,\det T_X\vert_U)$, which is defined by $d_\psi\mapsto d_\tau$.
We say that a trivialization $\tau$ is \emph{compatible} with the chosen coordinates $(\psi,U)$ if the image of the distinguished section $d$ under the canonical isomorphism $\rho:\sheafHom(\det T_X\vert_U,\det T_X\vert_U)\cong \mathcal{O}_U$ is a square \cite[Definition 21]{KW-cubic}.

Given a compatible trivialization $\tau:T_X\vert_U\cong U\times \A^n_k$, the section $\sigma$ trivializes to $\sigma:U\rightarrow \A^n_k$.
We can then define the \emph{local index} $\ind_x\sigma$ at $x$ to be the $\A^1$-degree of the composite
\begin{align*}
    \frac{\P^n_k}{\P^{n-1}_k}\to\frac{\P^n_k}{\P^n_k\backslash\{\psi(x)\}}\cong\frac{\A^n_k}{\A^n_k\backslash\{\psi(x)\}}\cong\frac{U}{U\backslash\{x\}}\xto{\sigma}\frac{\A^n_k}{\A^n_k\backslash\{0\}}\cong\frac{\P^n_k}{\P^{n-1}_k}.
\end{align*}
Here, the first map is the collapse map, the second map is excision, the third map is induced by the Nisnevich coordinates $(\psi,U)$, and the fifth map is purity (see e.g.~\cite[Definition 7.1]{BW}). By \cite[Theorem 3]{KW-cubic}, the $\A^1$-Euler characteristic is then the sum of local indices
\[\chi^{\A^1}(X)=\sum_{x\in\sigma^{-1}(0)}\ind_x\sigma\in \GW(k).\]
By Theorem~\ref{thm:main-thm}, we may thus compute the $\A^1$-Euler characteristic by computing the global B\'ezoutian bilinear form of an appropriate map $f:\A^n_k\to\A^n_k$.
\begin{remark}
If all the zeros of $f:\A^n_k\rightarrow \A^n_k$ are simple, then each local ring $Q_{\mathfrak{m}}$ in the decomposition of $Q=\frac{k[x_1,\ldots,x_n]}{(f_1,\ldots,f_n)}=Q_{\mathfrak{m_1}}\times\ldots\times Q_{\mathfrak{m}_s}$ is equal to the residue field of the corresponding zero. If each residue field $Q_{\mathfrak{m}_i}$ is a separable extension of $k$, then the $\A^1$-degree of $f$ is equal to sum of the scaled trace forms $\operatorname{Tr}_{Q_{\mathfrak{m}_i}/k}(\gw{J(f)|_{\mathfrak{m}_i}})$ (see e.g. \cite[Definition 1.2]{trace-paper}), where $J(f)|_{\mathfrak{m}_i}$ is the determinant of the Jacobian of $f$ evaluated at the point $\mathfrak{m}_i$. In \cite{pauli2020computing} the last named author uses the scaled trace form for several $\A^1$-Euler number computations. However, Theorem \ref{thm:main-thm} yields a formula for $\deg^{\A^1}(f)$ for any $f$ with only isolated zeros and without any restriction on the residue field of each zero. Moreover, we can even compute $\deg^{\A^1}(f)$ without solving for the zero locus of $f$.
\end{remark}

\subsection{The $\A^1$-Euler characteristic of Grassmannians}
Let $G:=\Gr_k(r,n)$ be the Grassmannian of $r$-planes in $k^n$. In order to compute $\chi^{\A^1}(G)$, we first need to describe Nisnevich coordinates and compatible trivializations for $G$ and $T_{G}$. We then need to choose a convenient section of $T_{G}$ and describe the resulting endomorphism $\A^{r(n-r)}_k$. The tangent bundle $T_{G}\to G$ is isomorphic to $p: \sheafHom(\mathcal{S},\mathcal{Q})\rightarrow G,$ where $\mathcal{S}\rightarrow G$ and $\mathcal{Q}\rightarrow G$ are the universal sub and quotient bundles.

We now describe Nisnevich coordinates on $G$ and a compatible trivialization of $T_{G}$, following \cite{SW}. Let $d=r(n-r)$ be the dimension of $G$, and let $\{e_1,\ldots,e_n\}$ be the standard basis of $k^n$. Let $\A^d_k=\Spec{k[\{x_{i,j}\}_{i,j=1}^{r,n-r}]}\cong U\subset G$ be the open affine subset consisting of the $r$-planes
\begin{align*}
H(\{x_{i,j}\}_{i,j=1}^{r,n-r}):=\operatorname{span}\left\{e_{n-r+i}+\sum_{j=1}^{n-r} x_{i,j}e_j\right\}_{i=1}^{r}.
\end{align*}
The map $\psi:U\to\A^d_k$ given by $\psi\left(H(\{x_{i,j}\}_{i,j=1}^{r,n-r})\right)=(\{x_{i,j}\}_{i,j=1}^{n-r,r})$ yields Nisnevich coordinates $(\psi,U)$ centered at $\psi\left(\operatorname{span}\{e_{n-r+1},\ldots,e_n\}\right)=(0,\ldots,0)$. For the trivialization of $T_{G}\vert_U$, let
\[
    \til{e}_i=\begin{cases}e_i & i\leq n-r,\\
    e_i+\sum_{j=1}^{n-r}x_{i-(n-r),j}e_j & i\geq n-r+1.\end{cases}
\]
Then $\{\til{e}_1,\ldots,\til{e}_n\}$ is a basis for $k^n$, and we denote the dual basis by $\{\til{\phi}_1,\ldots,\til{\phi}_n\}$. Over $U$, the bundles $\mathcal{S}^*$ and $\mathcal{Q}$ are trivialized by $\{\til{\phi}_{n-r+1},\ldots,\til{\phi}_n\}$ and $\{\til{e}_1,\ldots,\til{e}_{n-r}\}$, respectively. Since 
\[T_{G}\cong\sheafHom(\mathcal{S},\mathcal{Q})\cong\mathcal{S}^*\otimes\mathcal{Q},\]
we get a trivialization of $T_{G}\vert_U$ given by $\{\til{\phi}_{n-r+i}\otimes\til{e}_j\}_{i,j=1}^{r,n-r}$. By construction, our Nisnevich coordinates $(\psi,U)$ induce this local trivialization of $T_G$. It follows that the distinguished element of $\Hom(\det T_G|_U,\det T_G|_U)$ sending the distinguished element of $\det T_G|_U$ (determined by the Nisnevich coordinates) to the distinguished element of $T_G|_U$ (determined by our local trivialization) is just the identity, which is a square.

Next, we describe sections of $T_{G}\to G$ and the resulting endomorphisms $\A^d_k\to\A^d_k$. Let $\{\phi_1,\ldots,\phi_n\}$ be the dual basis of the standard basis $\{e_1,\ldots,e_n\}$ of $k^n$. A homogeneous degree 1 polynomial $\alpha\in k[\phi_1,\ldots,\phi_n]$ gives rise to a section $s$ of $\mathcal{S}^*$, defined by evaluating $\alpha$. In particular, given a vector $t=\sum_{i=1}^n t_i\til{e}_i$ in $H(\{x_{i,j}\}_{i,j=1}^{r,n-r})$, we use the dual change of basis
\[
    \phi_j=\begin{cases}
    \til{\phi}_j+\sum_{i=1}^rx_{i,j}\til{\phi}_{n-r+i} & j\leq n-r,\\
    \til{\phi}_j & j\geq n-r+1
    \end{cases}
\]
to set
\begin{align*}
    s(t)&=\alpha\left(t_1+\sum_{i=1}^r x_{i,1}t_{n-r+i},\ldots,t_{n-r}+\sum_{i=1}^r x_{i,n-r}t_{n-r+i},t_{n-r+1},\ldots,t_n\right).
\end{align*}
Note that $t_1=\cdots=t_{n-r}=0$ if and only if $t\in H(\{x_{i,j}\}_{i,j=1}^{r,n-r})$, so $s(t)\in k[t_{n-r+1},\ldots,t_n]$. Taking $n$ sections $s_1,\ldots,s_n$ of $\mathcal{S}^*$, we get a section of $T_{G}\cong\sheafHom(\mathcal{S},\mathcal{Q})$ given by
\[\mathcal{S}\xrightarrow{(s_1,\ldots,s_n)}\A^n_k\rightarrow\mathcal{Q},\]
where the second map is quotienting by $\{\til{e}_{n-r+1},\ldots,\til{e}_n\}$. We obtain our map $\A^d_k\to\A^d_k$ by applying the trivializations $\{\til{\phi}_{n-r+i}\otimes\til{e}_j\}_{i,j=1}^{r,n-r}$ of $T_{G}$. Explicitly, take $n$ sections $s_1,\ldots,s_n$ of $\mathcal{S}^*$. Since $e_i=\til{e}_i-\sum_{j=1}^{n-r}x_{i-(n-r),j}e_j$ for $i>n-r$, we have
\[s_je_j\equiv s_je_j-\sum_{i=1}^r x_{i,j}s_{n-r+i}e_j\Mod(\til{e}_{n-r+1},\ldots,\til{e}_n),\]
for all $j\leq n-r$. Recall that $e_j=\til{e}_j$ for $j\leq n-r$. The coordinate of $\A^d_k\to\A^d_k$ corresponding to $\til{\phi}_{n-r+i}\otimes\til{e}_j$ is thus the coefficient of $t_{n-r+i}$ in $s_j(t)-\sum_{\ell=1}^rx_{\ell,j}s_{n-r+\ell}(t)$.

For a general section $\sigma$ of $p:T_G\to G$, the finitely many zeros of $\sigma$ will all lie in $U$. In this case, the $\A^1$-Euler characteristic of $G$ is equal to the global $\A^1$-degree of the resulting map $\A_k^d\rightarrow \A^d_k$, which can computed using the B\'{e}zoutian. 

\begin{example}[$\Gr_k(2,4)$]\label{ex:grassmannian}
Let
\begin{align*}
    \alpha_1&=\phi_2=\til{\phi}_2+x_{1,2}\til{\phi}_3+x_{2,2}\til{\phi}_4,\\
    \alpha_2&=\phi_3=\til{\phi}_3,\\
    \alpha_3&=\phi_4=\til{\phi}_4,\\
    \alpha_4&=\phi_1=\til{\phi}_1+x_{1,1}\til{\phi}_3+x_{2,1}\til{\phi}_4.
\end{align*}
Evaluating at $t=(0,0,t_3,t_4)$ in the basis $\{\til{e}_i\}$, we have
\begin{align*}
    s_1&=x_{1,2}t_3+x_{2,2}t_4,\\
    s_2&=t_3,\\
    s_3&=t_4,\\
    s_4&=x_{1,1}t_3+x_{2,1}t_4.
\end{align*}
It remains to read off the coefficients of $t_3$ and $t_4$ of
\begin{align*}
    s_1-x_{1,1}s_3-x_{2,1}s_4&=(x_{1,2}-x_{1,1}x_{2,1})t_3+(x_{2,2}-x_{1,1}-x_{2,1}^2)t_4,\\
    s_2-x_{1,2}s_3-x_{2,2}s_4&=(1-x_{1,1}x_{2,2})t_3+(-x_{1,2}-x_{2,1}x_{2,2})t_4.
\end{align*}
We thus have our endomorphism $\sigma:\A^{4}_k\to\A^{4}_k$ defined by
\[\sigma=(x_{1,2}-x_{1,1}x_{2,1},x_{2,2}-x_{1,1}-x_{2,1}^2,1-x_{1,1}x_{2,2},-x_{1,2}-x_{2,1}x_{2,2}).\]
Using the Sage implementation of the B\'ezoutian formula for the $\A^1$-degree~\cite{code}, we can calculate  $\chi^{\A^1}(\Gr_k(2,4))=\deg^{\A^1}(\sigma)=2\HH+\gw{1,1}$.
\end{example}

Using a computer, we performed computations analogous to Example~\ref{ex:grassmannian} for $r\leq 5$ and $n\leq 7$. These $\A^1$-Euler characteristics of Grassmannians are recorded in Figure~\ref{fig:Euler-char-grassmannians}.

\begin{figure}
 \begin{tabular}{|c | c c c c c|} 
 \hline
 \tiny{\diagbox{$n$}{$r$}} & 1 & 2 & 3 & 4& 5\\ 
 \hline
 2 &  $\HH$ &$\langle 1\rangle$  &  &&\\ 
 
 3 & $\HH+\langle 1\rangle$ & $\HH+\langle 1\rangle$ &$\langle 1\rangle$& &\\
 
 4 & $2\HH$ & $2\HH+\langle 1, 1\rangle$ & $2\HH$&$\langle 1\rangle$ &\\
 
 5 & $2\HH+\langle 1\rangle$ & $4\HH+\langle1,1\rangle$ & $4\HH+\langle 1,1\rangle$ &$2\HH+\langle 1\rangle$&$\langle 1\rangle$\\

 6 & $3\HH$ & $6\HH+\langle1,1,1\rangle$ & $10\HH$ &$6\HH+\langle1,1,1\rangle$& $3\HH$\\
 
  7 & $3\HH+\langle1\rangle$ & $9\HH+\langle1,1,1\rangle$ & $16\HH+\langle1,1,1\rangle$ &$16\HH+\langle1,1,1\rangle$ & $9\HH+\langle1,1,1\rangle$ \\
 \hline

\end{tabular}
\caption{More examples of $\chi^{\A^1}\left(\Gr_k(r,n)\right)$}
\label{fig:Euler-char-grassmannians}
\end{figure}

Recall that the Euler characteristics of real and complex Grassmannians are given by binomial coefficients. In particular, these Euler characteristics satisfy certain recurrence relations related to Pascal's rule. The computations in Figure~\ref{fig:Euler-char-grassmannians} indicate that an analogous recurrence relation is true for the $\A^1$-Euler characteristic of Grassmannians over an arbitrary field. In fact, this recurrence relation is a direct consequence of a result of Levine~\cite{levine2020aspects}.

\begin{proposition}\label{prop:Gr-recurrence} 
Let $1\leq r<n$ be integers. Then
\[
    \chi^{\A^1}\left(\Gr_k(r,n)\right)=\chi^{\A^1}\left(\Gr_k(r-1,n-1)\right)+\gw{-1}^r\chi^{\A^1}\left(\Gr_k(r,n-1)\right).
\]
\end{proposition}
\begin{proof} 
Fix a line $L$ in $k^n$. Let $Z$ be the closed subvariety consisting of all $r$-planes containing $L$ (which is isomorphic to $\Gr_k(r-1,n-1)$), and let $U$ be its open complement (which is isomorphic to an affine rank $r$ bundle over $\Gr_k(r,n-1)$). We then get a decomposition $\Gr_k(r,n)=Z\cup U$. Since $\Gr_k(l,m)\cong\Gr_k(m-l,m)$, we have $\chi^{\A^1}(\Gr_k(l,m))=\chi^{\A^1}(\Gr_k(m-l,m))$. We can thus apply \cite[Proposition 1.4 (3)]{levine2020aspects} to obtain
\begin{align*}
    \chi^{\A^1}\left(\Gr_k(r,n)\right)&=\chi^{\A^1}\left(\Gr_k(n-r,n)\right)\\
    &=\chi^{\A^1}\left(\Gr_k(n-r,n-1)\right)+\gw{-1}^r\chi^{\A^1}\left(\Gr_k(n-r-1,n-1)\right)\\
    &=\chi^{\A^1}\left(\Gr_k(r-1,n-1)\right)+\gw{-1}^r\chi^{\A^1}\left(\Gr_k(r,n-1)\right).\qedhere
\end{align*}
\end{proof}

We can now apply a theorem of Bachmann--Wickelgren~\cite{BW} to completely characterize $\chi^{\A^1}(\Gr_k(r,n))$.

\begin{theorem}\label{thm:grassmannian}
Let $k$ be field of characteristic not equal to 2. Let $n_\mathbb{C} := \binom{n}{r}$, and let $n_\mathbb{R}:= 0$ if $n$ is even and $k$ is odd, and $n_\mathbb{R} = \binom{\floor{\frac{n}{2}}}{\floor{\frac{r}{2}}}$ otherwise. Then
\[\chi^{\A^1}(\Gr_k(r,n))=\frac{n_\mathbb{C} + n_\mathbb{R}}{2}\gw{1} + \frac{n_\mathbb{C} - n_\mathbb{R}}{2}\gw{-1}.\]
\end{theorem}
\begin{proof} 
By \cite[Theorem~5.8]{BW}, we can restrict this computation to two different possibilities. We will prove by induction that $\chi^{\A^1}(\Gr_k(r,n))\Mod\mathbb{H}$ has no $\gw{2}$ summand. The desired result will then follow from \cite[Theorem~5.8]{BW} by noting that $n_\mathbb{C}$ and $n_\mathbb{R}$ are the Euler characteristics of $\Gr_\mathbb{C}(r,n)$ and $\Gr_\mathbb{R}(r,n)$, respectively.

Since $\A^n_k$ is $\A^1$-homotopic to $\Spec{k}$, we have $\chi^{\A^1}(\A^n_k)=\chi^{\A^1}(\Spec{k})=\gw{1}$. Using this observation and the decomposition $\P^n_k=\bigcup_{i=0}^n\A^i_k$ (and a result analogous to \cite[Proposition 1.4 (3)]{levine2020aspects}), Hoyois computed the $\A^1$-Euler characteristic of projective space \cite[Example 1.7]{hoyois}:
\begin{align*}
    \chi^{\A^1}(\P^n_k) &= \begin{cases} \frac{n}{2}\mathbb{H} + \gw{1} & n\text{ is even}, \\ \frac{n+1}{2}\mathbb{H} & n\text{ is odd}. \end{cases}
\end{align*}
Note that $\Gr_k(0,n)\cong\Gr_k(n,n)\cong\Spec{k}$ and $\Gr_k(1,n)\cong\Gr_k(n-1,n)\cong\P^{n-1}_k$. In particular, $\chi^{\A^1}(\Gr_k(i,n))\Mod\mathbb{H}$ is either trivial or $\gw{1}$ for $i=0,1,n-1$, or $n$. This forms the base case of our induction, with the inductive step given by Proposition~\ref{prop:Gr-recurrence} -- namely, if $\chi^{\A^1}(\Gr_k(r-1,n-1))\Mod\HH$ and $\chi^{\A^1}(\Gr_k(r,n-1))\Mod\HH$ only have $\gw{1}$ and $\gw{-1}$ summands, then
\[(\chi^{\A^1}(\Gr_k(r-1,n-1))+\gw{-1}^r\chi^{\A^1}(\Gr_k(r,n-1)))\Mod\HH\]
only has $\gw{1}$ and $\gw{-1}$ summands.
\end{proof}

\subsection{Modified Pascal's triangle for $\chi^{\A^1}(\Gr_k(r,n))$}
Pascal's triangle gives a mnemonic device for binomial coefficients and hence for the Euler characteristics of complex and real Grassmannians. The recurrence relation of Proposition~\ref{prop:Gr-recurrence} indicates that a modification of Pascal's triangle can also be used to calculate the $\A^1$-Euler characteristics of Grassmannians. Elements of each row of the modified Pascal's triangle are obtained from the previous row by the addition rule illustrated in Figure~\ref{fig:addition rules}. We rewrite the data recorded in Figure~\ref{fig:Euler-char-grassmannians} in a modified Pascal's triangle in Figure~\ref{fig:pascal's triangle}.

\begin{samepage}
\begin{figure}[p]
\begin{tabular}{c}\begin{tikzpicture}[y=7.5mm,x=8.66mm]
  \colorlet{even}{cyan!60!black}
  \colorlet{lightcyan}{cyan!15!}
  \tikzset{
    box/.style={
      regular polygon,
      regular polygon sides=6,
      minimum size=20mm,
      inner sep=0mm,
      outer sep=0mm,
      shape border rotate=30,
      text centered,
      font={\fontsize{11pt}{0}\sffamily},
      draw=#1,
      line width=.25mm,
    },
  }
   \node[box=even,fill=lightcyan] at (-1,0) {$a$};
   \node[box=even] at (1,0) {$b$};
   \node[box=even] at (0,-2) {$a+b$};
\end{tikzpicture}\end{tabular}
\qquad
\begin{tabular}{c}\begin{tikzpicture}[y=7.5mm,x=8.66mm]
  \colorlet{even}{cyan!60!black}
  \colorlet{lightcyan}{cyan!15!}
  \tikzset{
    box/.style={
      regular polygon,
      regular polygon sides=6,
      minimum size=20mm,
      inner sep=0mm,
      outer sep=0mm,
      text centered,
      font={\fontsize{11pt}{0}\sffamily},
      draw=#1,
      line width=.25mm,
    },
    smallbox/.style={
      regular polygon,
      regular polygon sides=6,
      minimum size=20mm,
      inner sep=0mm,
      outer sep=0mm,
      text centered,
      font={\fontsize{7pt}{0}\sffamily},
      draw=#1,
      line width=.25mm,
    },
  }
   \node[box=even, shape border rotate=30] at (-1,0) {\(a\)};
   \node[box=even, shape border rotate=30,fill=lightcyan] at (1,0) {$b$};
   \node[smallbox=even, shape border rotate=30,fill=lightcyan] at (0,-2) {$a+ \gw{-1}b$};
\end{tikzpicture}\end{tabular}
\caption{Addition rules for modified Pascal's triangle}
\label{fig:addition rules}
\end{figure}

\begin{figure}[p]
\begin{tikzpicture}[y=7.5mm,x=8.66mm]
  \colorlet{even}{cyan!60!black}
  \colorlet{lightcyan}{cyan!15!}

  \newcommand\Nlabel{10}
  \tikzset{
    box/.style={
      regular polygon,
      regular polygon sides=6,
      minimum size=20mm,
      inner sep=0mm,
      outer sep=0mm,
      text centered,
      font={\fontsize{9pt}{0}\sffamily},
      draw=#1,
      line width=.25mm,
    },
    smallbox/.style={
      regular polygon,
      regular polygon sides=6,
      minimum size=20mm,
      inner sep=0mm,
      outer sep=0mm,
      text centered,
      font={\fontsize{7pt}{0}\sffamily},
      draw=#1,
      line width=.25mm,
    },
    tinybox/.style={
      regular polygon,
      regular polygon sides=6,
      minimum size=20mm,
      inner sep=0mm,
      outer sep=0mm,
      text centered,
      font={\fontsize{6pt}{0}\sffamily},
      draw=#1,
      line width=.25mm,
    },
    rowlabel/.style={
        font = {\fontsize{10pt}{0}\sffamily},
    }
  }
  \node[box=even, shape border rotate=30] at (0,0) {$\gw{1}$};
  \foreach \x in {1,...,7}
    \node[box=even, shape border rotate=30] at ({-\x},{-2*\x}) {$\gw{1}$};
  \foreach \y in {1,...,3}
    \node[box=even, shape border rotate=30] at (2*\y,-4*\y) {$\gw{1}$};
  \foreach \y in {0,...,3}
    \node[box=even, shape border rotate=30,fill=lightcyan] at (2*\y+1,-4*\y-2) {$\gw{1}$};

   \node[rowlabel,left] at (-\Nlabel,0) {$n=0$};
   \node[rowlabel,left] at (-\Nlabel,-2) {$n=1$};
   
   \node[rowlabel] at (1,2) {\rotatebox{60}{$r=0$}};
   \node[rowlabel] at (2,0) {\rotatebox{60}{$r=1$}};
   \node[rowlabel] at (3,-2) {\rotatebox{60}{$r=2$}};
   \node[rowlabel] at (4,-4) {\rotatebox{60}{$r=3$}};
   \node[rowlabel] at (5,-6) {\rotatebox{60}{$r=4$}};
   \node[rowlabel] at (6,-8) {\rotatebox{60}{$r=5$}};
   \node[rowlabel] at (7,-10) {\rotatebox{60}{$r=6$}};
   \node[rowlabel] at (8,-12) {\rotatebox{60}{$r=7$}};

   \node[rowlabel,left] at (-\Nlabel,-4) {$n=2$};
   \node[box=even, shape border rotate=30,fill=lightcyan] at (0,-4) {$\mathbb{H}$};

    \node[rowlabel,left] at (-\Nlabel,-6) {$n=3$};
    \node[box=even, shape border rotate=30,align=center,fill=lightcyan] at (-1,-6) {$\mathbb{H} + \gw{1}$};
    \node[box=even, shape border rotate=30] at (1,-6) {$\mathbb{H} + \gw{1}$};

    \node[rowlabel,left] at (-\Nlabel,-8) {$n=4$};
    \node[box=even, shape border rotate=30,align=center,fill=lightcyan] at (-2,-8) {$2\mathbb{H}$};
    \node[smallbox=even, shape border rotate=30,align=center] at (0,-8) {$2\mathbb{H} + 2 \gw{1}$};
    \node[box=even, shape border rotate=30,align=center,fill=lightcyan] at (2,-8) {$2\mathbb{H}$};

    \node[rowlabel,left] at (-\Nlabel,-10) {$n=5$};
    \node[box=even, shape border rotate=30,align=center,fill=lightcyan] at (-3,-10) {$2\mathbb{H} + \gw{1}$};
    \node[smallbox=even, shape border rotate=30,align=center] at (-1,-10) {$4\mathbb{H} + 2 \gw{1}$};
    \node[smallbox=even, shape border rotate=30,align=center,fill=lightcyan] at (1,-10) {$4\mathbb{H} + 2\gw{1}$};
    \node[box=even, shape border rotate=30,align=center] at (3,-10) {$2\mathbb{H} + \gw{1}$};

    \node[rowlabel,left] at (-\Nlabel,-12) {$n=6$};
    \node[box=even, shape border rotate=30,align=center,fill=lightcyan] at (-4,-12) {$3\mathbb{H}$};
    \node[smallbox=even, shape border rotate=30,align=center] at (-2,-12) {$6\mathbb{H} + 3 \gw{1}$};
    \node[box=even, shape border rotate=30,align=center,fill=lightcyan] at (0,-12) {$10\mathbb{H}$};
    \node[smallbox=even, shape border rotate=30,align=center] at (2,-12) {$6\mathbb{H} + 3 \gw{1}$};
    \node[box=even, shape border rotate=30,align=center,fill=lightcyan] at (4,-12) {$3\mathbb{H}$};

    \node[rowlabel,left] at (-\Nlabel,-14) {$n=7$};
    \node[box=even, shape border rotate=30,align=center,fill=lightcyan] at (-5,-14) {$3\mathbb{H} + \gw{1}$};
    \node[smallbox=even, shape border rotate=30,align=center] at (-3,-14) {$9\mathbb{H} + 3\gw{1}$};
    \node[tinybox=even, shape border rotate=30,align=center,fill=lightcyan] at (-1,-14) {$16\mathbb{H} + 3\hspace{-0.1em}\gw{1}$};
    \node[tinybox=even, shape border rotate=30,align=center] at (1,-14) {$16\mathbb{H} + 3\hspace{-0.1em}\gw{1}$};
    \node[smallbox=even, shape border rotate=30,align=center,fill=lightcyan] at (3,-14) {$9\mathbb{H} + 3\gw{1}$};
    \node[box=even, shape border rotate=30,align=center] at (5,-14) {$3\mathbb{H}+\gw{1}$};
\end{tikzpicture}
\caption{Modified Pascal's triangle for $\chi^{\A^1}(\Gr_k(r,n))$}
\label{fig:pascal's triangle}
\end{figure}
\end{samepage}

\printbibliography
\end{document}